\documentclass[12pt,reqno]{amsart}
\usepackage {amssymb,amsfonts}
\usepackage {amsmath}
\usepackage{amsthm}
\usepackage{mathtools}
\usepackage[old]{old-arrows} 
\usepackage{stmaryrd}

\usepackage[utf8]{inputenc}
\usepackage[english]{babel}

\usepackage{tikz}
\usepackage{tikz-cd}
\usetikzlibrary{babel} 
\usetikzlibrary{cd} 
\tikzset{
  symbol/.style={
    draw=none,
    every to/.append style={
      edge node={node [sloped, allow upside down,
      auto=false]{$#1$}}}
  }
}

\usepackage{graphicx}
\usepackage[alphabetic]{amsrefs}
\usepackage[bookmarks=false,hypertexnames=false,colorlinks, citecolor=blue]{hyperref}
\usepackage{bookmark}
\usepackage{verbatim,color,geometry}
\usepackage[colorinlistoftodos, textsize=small]{todonotes}
\makeatletter 
\@mparswitchfalse%
\makeatother
\reversemarginpar

\usepackage[shortlabels]{enumitem} 

\newtheorem{prop}{Proposition}[section]
\newtheorem{thm}[prop]{Theorem}
\newtheorem{defn-thm}[prop]{Theorem-Definition}
\newtheorem{defn-lem}[prop]{Lemma-Definition}
\newtheorem{cor}[prop]{Corollary}

\newtheorem{lem}[prop]{Lemma}
\newtheorem*{thm*}{Theorem}
\newtheorem*{cor*}{Corollary}

\theoremstyle{definition}

\newtheorem{rem}[prop]{Remark}

\newtheorem*{claim*}{Claim}
\newtheorem*{rem*}{Remark}

\newcommand{\su}{\subset}

\newcommand{\cx}{{\breve{x}}}
\newcommand{\cy}{{\breve{y}}}

\newcommand{\ttto}{\dashrightarrow} 
\newcommand{\tto}{\longrightarrow}

\newcommand{\mapstto}{\longmapsto}

\newcommand{\wt}{\widetilde}

\newcommand{\ov}{\overline}

\newcommand{\PP}{\mathbb{P}}

\newcommand{\AAA}{\mathbb{A}}

\newcommand{\OO}{\mathcal{O}}

\newcommand{\pp}{\mathfrak{p}}
\newcommand{\fq}{\mathfrak{q}}

\newcommand{\al}{\alpha}
\newcommand{\be}{\beta}

\newcommand{\ga}{\gamma}

\newcommand{\de}{\delta}

\newcommand{\sig}{\sigma}

\newcommand{\ka}{\kappa}

\newcommand{\eps}{\varepsilon}

\DeclareMathOperator{\Spec}{Spec}
\DeclareMathOperator{\Proj}{Proj}
\DeclareMathOperator{\divv}{div}

\begin{document}

\title[A canonical moving singularity]{Fibrations by plane quartic curves with \\
a canonical moving singularity}

\author{Cesar Hilario}
\address{Mathematisches Institut, Heinrich-Heine-Universit\"at, 40204 D\"usseldorf, Germany}
\email{cesar.hilario.pm@gmail.com}

\author{Karl-Otto Stöhr}
\address{IMPA, Estrada Dona Castorina 110, 22460-320 Rio de Janeiro, Brazil}
\email{stohr@impa.br}


\dedicatory{February 1, 2025}

\begin{abstract}
We classify fibrations by integral plane projective rational quartic curves whose generic fibre is regular but admits 
a non-smooth point that is a canonical divisor.
These fibrations can only exist in characteristic two.
The geometric generic fibre, which determines the generic behaviour of the special fibres, is an integral plane projective rational quartic curve over the algebraic closure of the function field of the base.
It has the remarkable property that the tangent lines at the non-singular points are either all bitangents or all non-ordinary inflection tangents;
moreover it is strange, that is, all the tangent lines meet in a common point.
We construct two fibrations that are universal in the sense that any other fibration with the aforementioned properties can be obtained from one of them by a base extension.
Furthermore, among these fibrations we choose a pencil of plane quartic curves and study in detail its geometry. We determine the corresponding minimal regular model and we describe it as a purely inseparable double covering of a quasi-elliptic fibration.
\end{abstract}

\maketitle

\setcounter{tocdepth}{1}

\section{Introduction}

By Bertini's theorem on variable singular points, in characteristic zero the general fibre of a fibration $f:T\to B$ of smooth algebraic varieties over an algebraically closed field $k$ is smooth.
This no longer holds in positive characteristic.
Classical examples include the quasi-elliptic fibrations that appeared in the classification of algebraic surfaces in characteristics 2 and 3 by Bombieri and Mumford \cite{BM76}.

In this paper we classify a certain class of fibrations with non-smooth general fibre, namely fibrations by plane projective rational quartic curves with a canonical moving singularity.
The generic fibre $C=f^{-1}(\eta)$ of such a fibration $f:T\to B$ is a regular non-hyperelliptic geometrically integral curve over the function field $K=k(B)$ of the base, which admits a non-smooth point $\pp$ that is a canonical divisor on $C$.
Moreover it is geometrically rational, i.e., the geometric generic fibre $C \otimes_K \ov K$ is a rational curve over the algebraic closure $\ov K$, and the arithmetic genus is equal to $h^1(C,\OO_C) = 3$.

It turns out that a curve $C$ satisfying the above properties can only exist in characteristic $p=2$. 
Our main result gives a precise description 

\begin{thm}[see Theorem~\ref{2023_04_10_16:50}] 
\label{2023_05_23_19:35}
Let $C$ be a regular proper non-hyperelliptic geometrically integral and geometrically rational curve of arithmetic genus $h^1(C,\OO_C)=3$ over a field $K$ of characteristic $p$.
Assume in addition that $C$ admits a non-smooth point $\pp$ that is a canonical divisor on $C$.
Then $p=2$ and $C$ is isomorphic to one of the following plane projective quartics. 
\begin{enumerate}[\upshape (i)]
   \item \label{2023_05_23_16:02}
   $y^4 + a z^4 + x z^3 + b x^2 z^2 + c x^4 = 0$, \\
   where $a,b,c \in K$ are constants satisfying $c \notin K^2$.
   \item \label{2023_05_23_16:01}
   $ y^4 + a z^4 + b x^2 y^2 + c x^2 z^2 + b x^3 z + d x^4 = 0$, \\
   where $a,b,c,d \in K$ are constants satisfying $a\notin K^2$ and $b\neq 0$.
   In this case the polynomial expression $\iota = a b^2 + c^2 + 1$ is an invariant of the curve $C$.
\end{enumerate}
Conversely, each of these equations defines a curve of the above type.
\end{thm}

In addition, we determine the isomorphism classes of these curves (see Propositions~\ref{2023_06_05_16:35} and~\ref{2024_08_11_01:30}), 
thus completing their classification.
The point $\pp$ is the only non-smooth point on the regular curve $C$.
Its images $\pp_n \in C_n$ in the sequence of normalized Frobenius pullbacks
\[ C_0 = C \to C_1 \to C_2 \to C_3 \to \cdots, \]
where $C_n$ denotes the normalization of the $n$th iterated Frobenius pullback $C^{(p^n)}$,
have geometric $\de$-invariants 
$\de(\pp)=3$, $\de(\pp_1)=1$, $\de(\pp_2) = \de(\pp_3) = \dots = 0$.
For $n\geq 3$ the smooth point $\pp_n$ is actually rational (see \cite[Theorem~2.24]{HiSt22}), whereas for $n=2$ both possibilities can occur: 
$\pp_2$ is rational for item~\ref{2023_05_23_16:02}, and non-rational for item~\ref{2023_05_23_16:01}.

If the curve $C|K$ is the generic fibre of a fibration $f:T\to B$, then the properties of almost all special fibres are governed by the properties of the geometric generic fibre $C\otimes_K \ov K$, 
which is an integral plane projective rational quartic curve over the algebraic closure $\ov K = \ov {k(B)}$.
Its only singularity is a unibranch point, whose only tangent line does not intersect the curve at any other point.
The quartic curve $C\otimes_K \ov K$ is \emph{strange}, i.e., all its tangent lines meet in a common point,
and it has the remarkable property that the tangent lines at its non-singular points are 
either all bitangents (in case~\ref{2023_05_23_16:01}) or all non-ordinary inflection tangents (in case~\ref{2023_05_23_16:02}).

We construct two fibrations by plane projective rational quartics with a canonical moving singularity, that are universal in the sense that any other fibration whose generic fibre satisfies the hypotheses in Theorem~\ref{2023_05_23_19:35} can be obtained from one of them by a base extension (see Theorems~\ref{2023_05_22_23:55} and~\ref{2023_05_23_16:35}).
We prove that their total spaces are uniruled, and in fact that the total space of every closed subfibration obtained from them is uniruled.

We show that the only singular point on $C\otimes_K \ov K$ has multiplicity two or three.
The case of multiplicity three occurs exactly in item~\ref{2023_05_23_16:02}, when $b=0$;
the corresponding pencil of quartics is studied in detail in Section~\ref{2024_07_01_19:25},
where we describe explicitly its minimal regular model
(see Theorem~\ref{2022_02_11_12:20}).
The determination of the minimal model is inspired by work of Kodaira and Néron \cite{Kod63,Ner64} 
on the classification of special fibres of minimal fibrations by elliptic curves.
Furthermore, we describe the pencil as a purely inseparable double cover of a \emph{quasi-elliptic fibration}.
The existence of such a cover is due to a remarkable fact: the normalized Frobenius pullbacks $C_1$ of the curves $C$ in Theorem~\ref{2023_05_23_19:35} are \emph{quasi-elliptic}.

Quasi-elliptic fibrations 
were introduced by Bombieri and Mumford in their extension of the Enriques-Kodaira classification of complex algebraic surfaces to arbitrary characteristics \cite{BM76,BM77}. 
The generic fibre of a quasi-elliptic fibration is a \emph{quasi-elliptic curve}, which is a regular geometrically integral curve of arithmetic genus $1$ that is non-smooth.
Since $C_1$ is quasi-elliptic, the curve $C$ is a purely inseparable cover of degree $p=2$ of a quasi-elliptic curve, via the Frobenius map $C \to C_1$.
This is a unique feature of geometry in characteristic $p=2$,
for in characteristic $p>2$ the normalized Frobenius pullback $X_1$ of any regular geometrically integral curve $X$ of arithmetic genus $h^1(X,\OO_X)=3$ is smooth (see \cite[Corollary~2.7]{HiSt22}).

The investigation of varieties over imperfect fields has attracted substantial attention in recent years due to their connection with the Minimal Model Program in positive characteristic, where Del Pezzo surfaces play a prominent role (see \cite{FaSc20, Tan21, BerTan22}). 
The case of interest to us, namely the case of regular but non-smooth curves, was analysed quite early by Tate \cite{Tate52}, who established restrictions on the drop in genus, by Queen \cite{Queen71}, who studied quasi-elliptic function fields, and quite recently by Tanaka \cite[Sections~9 and~10]{Tan21}, Fanelli and Schr\"oer \cite[Section~12]{FaSc20}, and Hilario and Schr\"oer \cite{HiSc23}.

The present article is the first in a sequence of two papers on the classification of fibrations by plane projective rational quartic curves in characteristic two (see~\cite{HiSt24}).
Our approach relies on the correspondence between function fields and regular curves \cite[II.7.4]{EGA}, a procedure we have also exploited in our previous article \cite{HiSt22}.
The starting point is the main theorem in \cite{HiSt22}, which guarantees that the function field of the regular curve $C_3|K$ is rational.
Then the Riemann--Roch theorem coupled with the Bedoya--St\"ohr algorithm \cite{BedSt87} provide us with suitably chosen algebraic functions that we can adjoin to the function field of $C_3|K$, to successively construct explicit presentations for the function fields of $C_2|K$, $C_1|K$ and $C|K$.
This allows us to realize $C|K$ as a plane projective quartic curve via a canonical embedding, given by the canonical linear system $|\pp|$.
In addition, since the regular curve $C_1|K$ is quasi-elliptic, on our way we are able to recover Queen's characterization of quasi-elliptic function fields \cite{Queen71} (see Section~\ref{2023_06_20_17:00}).

The paper is organized as follows.
In Section~\ref{2023_06_20_17:00} we obtain a description of quasi-elliptic function fields that is suited to our needs in Section~\ref{2023_04_10_16:30}, where we carry out a detailed study of the function fields $F|K$ of genus $g=3$ which are geometrically rational and have a unique singular prime $\pp$ that is a canonical divisor.
These results are employed in Section~\ref{2023_05_24_02:00} to construct their canonical models in the non-hyperelliptic case, 
which in turn yield the desired embeddings of our curves as plane projective quartics. 
In the last two sections we discuss the two aforementioned universal fibrations and the pencil of quartics that can be viewed as a cover of a quasi-elliptic pencil.

\subsection*{Acknowledgements}
We would like to thank Stefan Schröer for comments and feedback on an earlier version of this manuscript.
We also thank the referee for his/her very careful and thougthful reading, in particular for spotting several mistakes in Section~\ref{2023_06_20_17:00}.
Parts of this paper are based on the first author's PhD thesis, written with the financial support of CAPES Brazil.
The paper is also based on research performed in the framework of the research training group \textit{GRK 2240: Algebro-Geometric Methods in Algebra, Arithmetic and Topology}.

\section{Quasi-elliptic function fields}
\label{2023_06_20_17:00}
Given a regular proper geometrically integral curve $C$ over a field $K$, the properties of $C|K$ are encoded in the function field $F|K = K(C)|K$, which is a
\emph{one-dimensional separable function field}, i.e., $F|K$ is a separable finitely generated field extension of transcendence degree $1$, such that $K$ is algebraically closed in $F$.
Conversely, every one-dimensional separable function field $F|K$ is attached to a curve $C|K$ of the above type.

A one-dimensional separable function field $F|K$ is called \emph{quasi-elliptic} if it has genus $g=1$ and if it is geometrically rational, that is, the extended function field $\ov K F | \ov K = \ov K \otimes_K F |\ov K$ has genus $\ov g = 0$.
These function fields can only occur in characteristic $p \in \{2,3\}$, and they indeed play an important role in the arithmetic of Enriques surfaces in characteristic two (see for example \cite{BM76,KKM20}).

Quasi-elliptic function fields were investigated in detail by Queen~\cite{Queen71,Queen72}.
Our motivation for revising Queen's results in this section is to provide a description that is suitable for our purposes in Section~\ref{2023_04_10_16:30}, where we will encounter function fields $F|K$ whose first Frobenius pullbacks $F_1|K = F^p {\cdot} K|K$ are quasi-elliptic.
Our approach, based on the main theorem in \cite{HiSt22}, 
provides certain Riemann-Roch spaces that are not given in \cite{Queen71}, 
but which will be crucial to derive a normal form of $F|K$ out of a normal form of $F_1|K$ in Section~\ref{2023_04_10_16:30}.
The method also illustrates a general procedure to construct a normal form of a function field $F|K$ as soon as a normal form of some Frobenius pullback $F_n|K = F^{p^n} {\cdot} K|K$ is known (see \cite[p.\,281]{HiSt22}).
Recall that each Frobenius pullback $F_n|K$ is characterized by the property that $F_n$ is the only intermediate field of $F|K$ such that the extension $F|F_n$ is purely inseparable of degree $p^n$.

\medskip

Let $F|K$ be a quasi-elliptic function field. 
By Rosenlicht's genus drop formula (see \cite[Corollary of Theorem~11]{Ros52} or \cite[Formula~2.3]{HiSt22}) the number of the singular primes of $F|K$ counted with their geometric singularity degrees is equal to $g-\ov g = 1$, 
hence there is a unique singular prime $\pp$, whose geometric singularity degree $\de(\pp)$ is equal to $1$.
As singular primes can only occur in characteristic $p>0$,
in which case their singularity degrees are multiples of $(p-1)/2$ \cite[Proposition~2.5]{HiSt22}, we conclude that $p=2$ or $p=3$.

To study the extension of $\pp$ to the purely inseparable base field extension $F K^{1/p} |K^{1/p}$, one can apply the absolute Frobenius map $z\mapsto z^p$ and consider instead the restriction $\pp_1$ of $\pp$ to the first Frobenius pullback $F_1|K=F^p K |K$ of $F|K$.
The restricted prime $\pp_1$ is non-singular by \cite[Corollary~2.6]{HiSt22} and so by the genus drop formula the function field $F_1|K$ has genus $g_1=\ov g = 0$.
According to \cite[Corollary~2.19]{HiSt22} the prime $\pp$ is non-decomposed (in the base field extension $\ov K F|\ov K$), i.e., there is a unique prime of $\ov K F|\ov K$ lying over $\pp$,
i.e., $\pp$ is purely inseparable \cite[Corollary~2.17]{HiSt22}.
By \cite[Theorem~2.24]{HiSt22} the restriction $\pp_2$ of $\pp$ to the second Frobenius pullback $F_2|K$ is rational.
In characteristic $p=3$ the non-singular prime $\pp_1$ is also rational, but in characteristic $p=2$ it may not be, as illustrated in the following theorem.

\begin{thm}[{\cite[Theorem~2]{Queen71}}]\label{2023_05_26_02:00}
A one-dimensional separable function field $F|K$ of characteristic $p=2$ is quasi-elliptic, i.e., 
it is geometrically rational and admits a unique singular prime $\pp$, whose geometric singularity degree $\de(\pp)$ is equal to 1, 
if and only if $F|K$ can be put into one of the normal forms below.
\begin{enumerate}[\upshape (i)]
    \item \label{2023_05_26_02:01}
    $F|K=K(x,z)|K$, where 
    $$ z^2 = a_0 + x + a_2 x^2 + a_4 x^4$$
    and $a_0,a_2\in K$, $a_4 \in K\setminus K^2$.
    Here the singular prime $\pp$ has degrees $\deg(\pp)=2$, $\deg(\pp_1)=1$, and residue fields $\ka(\pp)=K(a_4^{1/2})$, $\ka(\pp_1)=K$.
    
    \item \label{2023_05_26_02:02}
    $F|K=K(x,z)|K$, where 
    $$z^4 = a_0 + x + a_2 x^2 + b_2^2 x^4,$$
    and $a_0\in K$, \mbox{$a_2 \in K \setminus K^2$}, \mbox{$b_2 \in K \setminus K^2(a_2)$}.
    Here $\deg(\pp)=4$, $\deg(\pp_1)=2$, $\deg(\pp_2)=1$, and $\ka(\pp)=K(a_2^{1/2},b_2^{1/2})$, $\ka(\pp_1) = K(a_2^{1/2})$, $\ka(\pp_2)=K$.
    
    \item \label{2023_05_26_02:03}
    $F|K=K(x,z)|K$, where 
    $$z^4 = a_0 + x + a_2 x^2, $$ 
    and $a_0\in K$, $a_2 \in K \setminus K^2$.
    Here $\deg(\pp)=\deg(\pp_1)=2$, $\deg(\pp_2)=1$, and $\ka(\pp) = \ka(\pp_1) = K(a_2^{1/2})$, $\ka(\pp_2)=K$.
\end{enumerate}
In each case the singular prime $\pp$ is the only pole of the function $x$.
\end{thm}

Before getting to the proof, 
let us note that although the standing assumption in \cite{BedSt87} is that the base field $K$ of $F|K$ is separably closed,
all we need in order to run the Bedoya-Stöhr algorithm is that the restriction $\pp_n$ of the prime $\pp$ in question is rational for some $n$.
This means that the prime $\pp$ is non-decomposed, a condition that is automatic if $K$ is separably closed (see \cite[Corollary~2.17]{HiSt22}).

\begin{proof}
Let $F|K$ be a quasi-elliptic function field and let $\pp$ be its only singular prime.
By the preceding discussion, the first Frobenius pullback $F_1|K$ has genus $g_1=0$ and the restricted primes $\pp_1$ and  $\pp_2$ are non-singular and rational respectively.

Assume first that $\pp_1$ is rational.
Then the prime $\pp$ is unramified over $F_1$, that is, $\deg(\pp) = p \, \deg(\pp_1) = 2$.
Since the Frobenius pullback $F_1|K$ has genus $g_1=0$ and its prime $\pp_1$ is rational, it is a rational function field, say $F_1|K=K(x)|K$ with $v_{\pp_1}(x)=-1$. 
By Riemann's theorem, the space of global sections $H^0(\pp_1^n)$ of the divisor $\pp_1^n$ has dimension $n+1$, and so
    \begin{equation*}
        H^0(\pp_1^n) = K \oplus K x \oplus\dots \oplus K x^n \ \text{ for all $n\geq 0$.}
    \end{equation*}
    Similarly, since the function field $F|K$ has genus $g=1$ and $\pp$ has degree $2$ we have
    \[ \dim H^0(\pp^n)=2n \  \text{ for all $n\geq1$.}  \]
    As $\dim H^0(\pp^2) = 4 >\dim H^0(\pp_1^2) = 3$, there is a function $z\in F$ such that
    \begin{equation*}
       H^0(\pp^2) = H^0(\pp_1^2) \oplus Kz = K\oplus Kx \oplus Kx^2 \oplus Kz.
    \end{equation*}
    This function does not belong to $F_1=K(x)$ because $H^0(\pp^2) \cap F_1 =  H^0(\pp_1^2)$. 
    Equivalently $F = F_1(z) = K(x,z)$, which means that $z$ is a separating variable of $F|K$, 
    i.e., the finite extension $F|K(z)$ is separable.
    Since the square $z^2$ lies in $ H^0(\pp^4)\cap F_1 = H^0(\pp_1^4)$, there exist constants $a_i \in K$ such that
    \[ z^2 = a_0 + a_1 x + a_2 x^2 + a_3 x^3 + a_4 x^4.   \]
    Note that one of the constants $a_1,a_3$ must be non-zero, for $x$ is separable over $K(z)$.
    Looking at $z^2$ as a Laurent series in the local parameter $x^{-1}$ at $\pp_1$, we deduce from
    \cite[Proposition~4.1]{BedSt87} that the non-rationality of $\pp$ together with $\de(\pp)=1$ mean that
    $a_3=0$, $a_1\neq 0$ and $a_4\notin K^2$,
    in which case we may normalize $a_1=1$ by substituting $x$ with $a_1 x$ and $z$ with $a_1 z$, respectively. This yields the normal form in~\ref{2023_05_26_02:01}.
    
    Conversely, since a singular prime satisfies the Jacobian criterion \cite[Corollaries~4.5 and~4.6]{Sal11}, the normal form in question guarantees that a function field $F|K = K(x,z)|K$ given as in~\ref{2023_05_26_02:01} has a unique singular prime, 
    namely the pole $\pp$ of $x$, which is inertial over $F_1=K(x)$ and has geometric singularity degree $\de(\pp)=1$, 
    i.e., $F|K$ is quasi-elliptic.
    
\medskip
We assume now that the non-singular prime $\pp_1$ is non-rational. 
Since $\pp_2$ is rational the prime $\pp_1$ is unramified over $F_2$ and has $\deg (\pp_1) = 2$.
Let $e$ denote the ramification index of $\pp$ over $F_1$.
As the function  field $F_2|K$ has genus $g_2 = 0$ and its prime $\pp_2$ is rational, we infer as before that $F_2|K$ takes the form $F_2|K = K(x)|K$ with $v_{\pp_2}(x) = -1$, and that 
\begin{equation*}
    H^0(\pp_2^n) = K \oplus Kx \oplus \dots \oplus K x^n \quad \text{for all $n\geq0$.}
\end{equation*}
By Riemann's theorem, since the function fields $F_1|K$ and $F|K$ have genera $g_1=0$ and $g=1$, and the divisors $\pp_1$ and $\pp^e$ have degrees $2$ and $4$, respectively, we deduce that
\begin{align*}
    \dim H^0(\pp_1^n) &= 2n + 1 \quad \text{for all $n\geq0$,} \\
    \dim H^0(\pp^{n e}) &= 4n \quad  \text{for all $n\geq 1$.}
\end{align*}
As $\dim H^0(\pp_1) = 3 > \dim H^0(\pp_2)=2$, we can find a function $w \in F_1$ such that
\begin{equation*}
    H^0(\pp_1) = H^0(\pp_2) \oplus Kw = K \oplus Kx \oplus Kw,
\end{equation*}
which does not belong to $F_2 = K(x)$ because $H^0(\pp_1) \cap F_2 = H^0(\pp_2)$. Thus $w$ is a separating variable of $F_1|K$, or equivalently $F_1 = F_2(w) = K(x,w)$. Since $w^2$ lies in $H^0(\pp_1^2)\cap F_2 = H^0(\pp_2^2)$, there exist constants $a_i \in K$ such that
$ w^2 = a_0 + a_1 x + a_2 x^2,$
with $a_1\neq 0$ because $x$ is separable over $K(w)$.
Hence we can normalize $a_1 = 1$ by replacing $x$ with $a_1^{-1} x$.
By \cite[Proposition~4.1]{BedSt87}, the fact that $\pp_1$ is 
non-rational and non-singular means that 
$a_2 \notin K^2$,
thus yielding the following normal form of $F_1|K$
\[ F_1|K = K(x,w)|K, \quad \text{where} \quad w^2 = a_0 + x + a_2 x^2, \quad \text{ $a_0 \in K, a_2\in K\setminus K^2$.} \]
Conversely, we observe that the quadratic equation already ensures that the function field $F_1|K = K(x,w)|K$ has genus $g_1 = 0$, and also that the pole $\pp_1$ of $x$ is non-rational with residue field $\ka(\pp_1)=K(a_2^{1/2})$.

Having found a normal form for $F_1|K$ we now proceed to find a normal form for $F|K$.
As $\dim H^0(\pp^{e})=4 > \dim H^0(\pp_1)=3$ there is a function $z \in F$ such that
\begin{equation*}
    H^0(\pp^{e}) = H^0(\pp_1) \oplus Kz = K \oplus Kx \oplus Kw \oplus Kz,
\end{equation*}
which lies outside $F_1 = K(x,w)$ because $H^0(\pp^{e}) \cap F_1 = H^0(\pp_1)$. 
In particular, $z$ is a separating variable of $F|K$, i.e., $F = F_1 (z) = K(x,w,z)$. 
Because $z^2$ belongs to $H^0(\pp^{2e}) \cap F_1 = H^0 (\pp_1^2) = K \oplus Kx \oplus Kx^2 \oplus Kw \oplus Kxw$, 
there are constants $b_i \in K$ such that
\[ z^2 = b_0 + b_1 x + b_2 x^2 + b_3 w + b_4 x w. \]	
One of the constants $b_3,b_4$ must be non-zero, since $z^2$ is a separating variable of $F_1|K$ and therefore $z^2 \notin F_2 = K(x)$.

We rephrase the fact that $\pp$ has geometric singularity degree $\de(\pp)=1$ 
in terms of equations on the constants $a_i,b_i$. 
To do this we introduce the functions $\cx := x^{-1} \in F_2$, $\breve{w} := w x^{-1} \in F_1$ and $\breve{z} := z x^{-1} \in F$. Note that $\cx$ is a local parameter at both $\pp_1$ and $\pp_2$, and that the separating variables $\breve{w}$ and $\breve{z}$ satisfy the relations
\[ \breve{w}^2 = a_2 + \cx + a_0 \cx^2, \qquad \breve{z}^2 = b_2 + b_1 \cx + b_0 \cx^2 + (b_4 + b_3 \cx)\breve{w}. \]
In particular, for the residue classes $\breve{w}(\pp),\breve{z}(\pp) \in \ka(\pp)$ we have
\[ 
\breve{w}(\pp)^2 = a_2 \notin K^2, \qquad \breve{z}(\pp)^2 = b_2 + b_4 \breve{w}(\pp). 
\]
We claim that the condition $\de({\pp}) = 1$ means that $b_4 = 0$ and $b_3 \neq 0$. Indeed, when $\breve{z}(\pp)\notin\ka({\pp_1})=K(\breve{w}(\pp))$ it suffices to observe that $\pp$ is inertial over $F_1$ and $\de({\pp}) = \frac12 v_{\pp_2}(d\breve{z}^4) = \frac12 v_{\pp_2}(b_4^2 + b_3^2 \cx^2)$, by \cite[Theorem~2.3]{BedSt87},
where we have expanded the differential $d\breve z^4$ of $F_2|K$ with respect to the local parameter $\breve x$ at $\pp_2$.
In the opposite case $\breve{z}(\pp) \in K(\breve{w}(\pp))$, say $t(\pp) = 0$ for some $t$ in $\breve{z} + K + K\breve{w}$, we have $b_4 = 0$ (and therefore $b_3 \neq 0$) since $\breve{w}(\pp) \notin K$, hence the prime $\pp$ is ramified over $F_1$ with local parameter $t$ because
\[ v_{\pp_2}(dt^4) = v_{\pp_2}(d\breve{z}^4) = v_{\pp_2}(b_3^2 \cx^2) = 2 < 4, \]				
and therefore $\de({\pp}) = \frac12 v_{\pp_2}(dt^4) = 1$ by \cite[Theorem~2.3]{BedSt87}, thus proving our claim.
Note that this implies in particular that $\ka(\pp) = K(a_2^{1/2},b_2^{1/2})$.

Substituting $x,w$ with $b_3^{-2} x,b_3^{-1} w$ we can normalize $b_3=1$, so we obtain the following normal form of $F|K$
\begin{equation}\label{2024_03_18_13:10}
    F|K = K(x,w,z)|K,  \quad \text{where} \quad w^2 = a_0 + x + a_2 x^2, \,   z^2 = b_0 + b_1 x + b_2 x^2 + w.
\end{equation}
We claim that this normal form already ensures that $F|K$ is quasi-elliptic.
Indeed, if $F|K = K(x,w,z)|K$ is given as above, then it follows from 
$z^4 = (a_0 + b_0^2) + x + (a_2 + b_1^2) x^2 + b_2^2 x^4$ 
and the Jacobian criterion that the pole $\pp$ of $x$ is the only singular prime of $F|K$, and therefore that $F|K$ is quasi-elliptic.

To get the normal forms in~\ref{2023_05_26_02:02} and~\ref{2023_05_26_02:03} we look at the residue fields $\ka(\pp_1) = K(a_2^{1/2})$, $\ka(\pp) = K(a_2^{1/2},b_2^{1/2})$.
If $e=1$, i.e., $b_2 \notin K^2(a_2)$, then we normalize $b_0 = b_1 = 0$ by replacing $w$ with $w + b_0 + b_1 x$, thus obtaining the normal form in~\ref{2023_05_26_02:02}.
In the opposite case $e=2$, i.e., $b_2 \in K^2 + K^2 a_2$, we normalize $b_2=0$ by subtracting from $z$ an element of $Kx + Kw$, and then $b_0 = b_1 = 0$ by subtracting $b_0 + b_1 x$ from $w$;
this gives the normal form in~\ref{2023_05_26_02:03}.
\end{proof}

\begin{rem}\label{2023_12_21_22:00}
Item~\ref{2023_05_26_02:01} in Theorem~\ref{2023_05_26_02:00} corresponds to the case in which the restricted prime $\pp_1$ is rational, whereas items~\ref{2023_05_26_02:02} and~\ref{2023_05_26_02:03} correspond to the case where $\pp_1$ is non-rational.
As explained in the above proof, in the latter situation the quasi-elliptic function field $F|K$ admits the unified normal form
\eqref{2024_03_18_13:10},
where $a_i,b_i\in K$ and $a_2\notin K^2$.
In this case the singular prime $\pp$ is the only pole of the function $x$ and its residue fields are equal to $\ka(\pp) = K(a_2^{1/2},b_2^{1/2})$, $\ka(\pp_1) = K(a_2^{1/2})$, $\ka(\pp_2)=K$.
\end{rem}

\begin{rem}\label{2023_12_21_22:35}
The proof of the theorem
supplies the Riemann-Roch spaces associated to the only singular prime $\pp$, 
which will be essential in the next section.
Let $F|K = K(x,z)|K$ be a function field as in Theorem~\ref{2023_05_26_02:00}~\ref{2023_05_26_02:01}. 
Then the first Frobenius pullback $F_1|K=K(x)|K$ is a rational function field
and the Riemann-Roch spaces associated to $\pp$ are given by
\begin{align*}
    H^0(\pp_1^n) &= K \oplus K x \oplus \cdots \oplus K x^n \quad \text{for all $n\geq 0$}, \\
    H^0(\pp^n) &= H^0(\pp_1^n) \oplus H^0(\pp_1^{n-2}){\cdot} z \quad \text{for all $n\geq 0$}.
\end{align*}
Now let $F|K$ be a function field as in item~\ref{2023_05_26_02:02} or item~\ref{2023_05_26_02:03}, i.e., let $F|K=K(x,z)|K$ be given by $z^4 = a_0 + x + a_2 x^2 + b_2^2 x^4$, where $b_2=0$ for item~\ref{2023_05_26_02:03}.
Set $w := z^2 + b_2 x^2$.
Then 
the first Frobenius pullback $F_1|K = K(x,w)|K$ has genus $g_1=0$
and 
the second Frobenius pullback $F_2|K=K(x)|K$ is rational.
The Riemann-Roch spaces take the form
\begin{align*}
    H^0(\pp_2^n) &= K \oplus K x \oplus \cdots \oplus K x^n \quad \text{for all $n\geq 0$}, \\
    H^0(\pp_1^n) &= H^0(\pp_2^n) \oplus H^0(\pp_2^{n-1}) {\cdot} w \quad \text{for all $n\geq0$,} \\
    H^0(\pp^{ne}) &= H^0(\pp_2^n) \oplus H^0(\pp_2^{n-1}) {\cdot} w \oplus H^0(\pp_2^{n-1}) {\cdot} z \oplus H^0(\pp_2^{n-2}) {\cdot} w z  \quad \text{for all $n\geq 0$},
\end{align*}
where $e$ denotes the ramification index of the extension $\pp|\pp_1$, i.e., $e=1$ for~\ref{2023_05_26_02:02}, and $e=2$ for~\ref{2023_05_26_02:03}.
If item~\ref{2023_05_26_02:03} occurs, i.e., $e=2$, i.e., $b_2=0$, i.e., $w = z^2$, then we also have
\[ H^0(\pp^{2n+1}) = H^0(\pp_2^n) \oplus H^0(\pp_2^n) {\cdot} z \oplus H^0(\pp_2^{n-1}) {\cdot} z^2 \oplus H^0(\pp_2^{n-1}) {\cdot} z^3 \quad\text{for all $n\geq 0$}, \]
because $z^2 = w \in H^0(\pp^2)$ and $\dim H^0(\pp) = \deg(\pp) = 2$, i.e., $H^0(\pp) = K \oplus Kz$.
Note that the assertions on the Frobenius pullbacks $F_1|K$, $F_2|K$ and the Riemann-Roch spaces $H^0(\pp_2^n)$, $H^0(\pp_1^n)$, $H^0(\pp^{ne})$ 
apply verbatim to 
the unified normal form \eqref{2024_03_18_13:10} in Remark~\ref{2023_12_21_22:00}.
\end{rem}

Arguing as in 
Theorem~\ref{2023_05_26_02:00} we recover Queen's theorem in characteristic $p=3$.

\begin{thm}[{\cite[Theorem~2]{Queen71}}]\label{2024_04_16_15:20}
A one-dimensional separable function field $F|K$ of characteristic $p=3$ is quasi-elliptic, i.e., it is geometrically rational and admits a unique singular prime $\pp$, whose geometric singularity degree $\de(\pp)$ is equal to $1$, if and only if $F|K$ can be put into the following normal form
\[ F|K=K(x,z)|K, \qquad z^3 = a_0 + x + a_3 x^3, \quad \text{where $a_0\in K,a_3\in K\setminus K^3$.}\]    
The singular prime $\pp$ is the pole of the function $x$.
The Frobenius pullback $F_1|K=K(x)|K$ is rational and the singular prime $\pp$ has degrees $\deg(\pp)=3$, $\deg(\pp_1)=1$, and residue fields 
$\ka(\pp) = K(a_3^{1/3})$, $\ka(\pp_1) = K$.
The Riemann-Roch spaces are given by
\begin{align*}
    H^0(\pp_1^n) &= K \oplus K x \oplus \dots \oplus K x^n  \quad \text{for all $n\geq0$}, \\
    H^0(\pp^n) &= H^0(\pp_1^n) \oplus H^0(\pp_1^{n-1}) {\cdot} z \oplus H^0(\pp_1^{n-2}) {\cdot} z^2 \quad \text{for all $n\geq0$}.
\end{align*}
\end{thm}

With the Riemann-Roch spaces we can determine the isomorphisms classes of all quasi-elliptic function fields.

\begin{prop}\label{2024_08_10_15:40}
Let $F|K$ and $F'|K$ be two quasi-elliptic function fields.
If both are defined as in item~\ref{2023_05_26_02:01}, \ref{2023_05_26_02:02} or~\ref{2023_05_26_02:03} in Theorem~\ref{2023_05_26_02:00},
then the $K$-isomorphisms $F' \overset\sim\to F$ are given by the transformations
\begin{enumerate}[\upshape (i)]
    \item $(x',z') \mapsto (\eps^2 (x + \al), \eps (z + \beta + \ga x + \tau x^2))$,
    \\
    where $\eps,\al,\be,\ga,\tau \in K$ are constants such that 
    $\eps\neq 0$,
    $\eps^2 a_2' = a_2 + \ga^2$,
    $\eps^6 a_4' = a_4 + \tau^2$,
    $\eps^{-2} a_0' = a_0 + \be^2 + \al + \al^2 (a_2 + \ga^2) + \al^4 (a_4 + \tau^2)$;
    
    \item \label{2024_08_10_15:42}
    $(x',z') \mapsto (\eps^4 (x + \al),\eps (z + \beta + \ga x + \tau z^2 + \tau b_2 x^2))$,
    \\
    where $\eps,\al,\be,\ga,\tau \in K$ are constants such that 
    $\eps\neq 0$,
    $\eps^4 a_2' = a_2 + \tau^4$, 
    $\eps^6 b_2' = b_2 + \ga^2 + \tau^2 a_2$,
    $\eps^{-4} a_0' = a_0 + \tau^4 a_0^2 + \be^4 + \al + \al^2 (a_2 + \tau^4) + \al^4 (b_2^2 + \tau^4 a_2^2 + \ga^4) $;

    \item \label{2024_08_10_15:43}
    $(x',z') \mapsto (\eps^4 (x + \al),\eps (z + \beta ))$,
    \\
    where $\eps,\al,\be\in K$ are constants such that 
    $\eps\neq 0$, 
    $\eps^{-4} a_0' = a_0 + \be^4 + \al + \al^2 a_2$,
    $\eps^4 a_2' = a_2$.
\end{enumerate}
If $F|K$ and $F'|K$ are defined as in Theorem~\ref{2024_04_16_15:20},
then the $K$-isomorphisms $F' \overset\sim\to F$ are given by
\[ (x',z') \mapsto (\eps^3 (x + \al), \eps (z + \beta + \ga x)), \]
where $\eps,\al,\be,\ga \in K$ are constants such that $\eps\neq 0$,
$\eps^{-3} a_0' = a_0 + \be^3 - \al - \al^3 (a_3 +\ga^3)$,
$\eps^{6} a_3' = a_3 + \ga^3$.

In particular, $F|K$ and $F'|K$ are $K$-isomorphic if and only if such constants do exist.
\end{prop}

We will only prove the second item,
the proofs of the remaining parts being similar.
Note that each $K$-isomorphism $F' \overset\sim\to F$ preserves the only singular primes $\pp'$ and $\pp$ of $F'|K$ and $F|K$, and also the ramification indices of $\pp_n'|\pp_{n+1}'$ and $\pp_n|\pp_{n+1}$ for all $n$.
Thus if $F'|K$ is given as in~\ref{2023_05_26_02:01}, \ref{2023_05_26_02:02} or~\ref{2023_05_26_02:03} in Theorem~\ref{2023_05_26_02:00}, then so is $F|K$.

\begin{proof}[Proof of~\ref{2024_08_10_15:42}]
Let $\sig:F' \overset\sim\to F$ be a $K$-isomorphism.
As $\sig$ preserves the only singular primes $\pp'$ and $\pp$ of $F'|K$ and $F|K$ respectively, it induces an isomorphism $H^0(\pp_m'^n) \overset\sim\to H^0(\pp_m^n)$ for each $m$ and $n$.
Thus the incidence properties of $x'$ and $z'$ inherit as follows
\begin{align*}
    \sig(x') \in H^0(\pp_2) \setminus K , \quad
    \sig(z') \in H^0(\pp) \setminus H^0(\pp_1).
\end{align*}
Moreover, the functions $\sig(x')$ and $\sig(z')$ also satisfy the polynomial equation with the coefficients $a_i'$, $b_i'$.
In this equation we substitute $\sig(x')$ and $\sig(z')$ by the corresponding $K$-linear combinations of $1,x$ and $1,x,z^2 + b_2 x^2,z$ respectively, and we replace $z^4$ with the right-hand side of the equation in the announcement of Theorem~\ref{2023_05_26_02:00}~\ref{2023_05_26_02:02}.
As the four functions $1$, $x$, $x^2$, $x^4$ are $K$-linearly independent, we obtain $4$ polynomial equations between $a_i$, $b_i$, $a_i'$, $b_i'$ and the $2 + 4 = 6$ coefficients of the expansions of $\sig(x')$ and $\sig(z')$.
\end{proof}

The proposition immediately yields the set of all automorphisms of every quasi-elliptic function field.
Note that the substitutions we obtain are further simplified by the condition that the constant $a_4$ (resp. $a_2$) 
in Theorem~\ref{2023_05_26_02:00}~\ref{2023_05_26_02:01} (resp. Theorem~\ref{2023_05_26_02:00}, \ref{2023_05_26_02:02} and \ref{2023_05_26_02:03}) is not a square.

\begin{cor}
Let $F|K$ be a quasi-elliptic function field.
If $F|K$ is defined as in Theorem~\ref{2023_05_26_02:00}, items~\ref{2023_05_26_02:01}, \ref{2023_05_26_02:02} or~\ref{2023_05_26_02:03}, then the automorphisms of $F|K$ are given by
\begin{enumerate}[\upshape (i)]
    \item $(x,z) \mapsto (\eps^2(x + \al), \eps (z + \beta + \ga x) )$, 
    \\
    where $\eps,\al,\be,\ga \in K$ satisfy $\eps^3 = 1$, $\ga^2 = (\eps^2 + 1) a_2 $, $\be^2 = (\eps + 1) a_0 + \al + \al^2 (a_2 + \ga^2) + \al^4 a_4$.
    \item $ (x,z) \mapsto ( x + \al, z + \beta)$, 
    \\
    where $\al,\be \in K$ satisfy $\be^4 = \al + \al^2 a_2 + \al^4 b_2^2$.
    \item $ (x,z) \mapsto ( x + \al, z + \beta)$, 
    \\
    where $\al,\be \in K$ satisfy $\be^4 = \al + \al^2 a_2$.
\end{enumerate}
If $F|K$ is defined as in Theorem~\ref{2024_04_16_15:20} then the automorphisms of $F|K$ take the form
\[ (x,z) \mapsto (\eps (x + \al), \eps (z + \beta )),\]
where $\eps,\al,\be \in K$ satisfy $\eps^2 = 1$ and $\be^3 = (\eps - 1) a_0 + \al + \al^3 a_3$.
\end{cor}

\begin{rem}
Let $C|K$ be a quasi-elliptic curve, i.e., let $C$ be a regular proper geometrically integral curve over $K$ of genus $g=1$, which is non-smooth.
It is known that $C$ can be embedded in $\PP^2(K)$ or $\PP^3(K)$ (see e.g. \cite[Theorem~11.13]{Tan21}).
The Riemann-Roch spaces allow for explicit realizations.
Let $\pp$ be the only non-smooth point on $C$.
In characteristic $p=3$ the function field $F|K = K(C)|K$ satisfies Theorem~\ref{2024_04_16_15:20}, hence $\deg(\pp)=3 \geq 2g+1$ and $H^0(\pp)=K \oplus Kx \oplus Kz$, so the complete linear system $|\pp|$ is very ample and it realizes $C$ as the plane curve 
\[ z^3 = a_0 y^3 + x y^2 + a_3 x^3, \quad \text{where $a_i\in K$, $a_3\notin K^3$.}\]
In characteristic $p=2$ there are two possibilities:
if $F|K$ satisfies Theorem~\ref{2023_05_26_02:00}~\ref{2023_05_26_02:01}, then $\deg(\pp^2)=4\geq 2g+1$ and $H^0(\pp^2)=K \oplus Kx \oplus Kx^2 \oplus Kz$, and so $|\pp^2|$ realizes $C$ as the complete intersection of the following two quadric surfaces in $\PP^3(K) = \Proj K[x_0,x_1,x_2,x_3]$
\[ x_1^2 = x_0 x_2, \quad x_3^2 = a_0 x_0^2 + x_0 x_1 + a_2 x_0 x_2 + a_4 x_2^2, \quad \text{where $a_i\in K$, $a_4\notin K^2$;} \]
if $F|K$ is given in normal form~\eqref{2024_03_18_13:10} (see Remark~\ref{2023_12_21_22:00}), with $a_i,b_i\in K$ and $a_2\notin K^2$, then $\deg(\pp^e)=4$ and $H^0(\pp^e)=K \oplus Kx \oplus Kw \oplus Kz$ ($e$ is the ramification index of $\pp|\pp_1$), hence $|\pp^e|$ realizes $C$ as the complete intersection of the two quadric surfaces in $\PP^3(K)$ below
\[ x_2^2 = a_0 x_0^2 + x_0 x_1 + a_2 x_1^2, \quad x_3^2 = b_0 x_0^2 + b_1 x_0 x_1 + b_2 x_1^2 + x_0 x_2, \quad \text{where $a_i,b_i\in K$, $a_2\notin K^2$.}\]
\end{rem}

\begin{rem}
Every quasi-elliptic function field $F|K$ with a distinguished rational point $\fq$ admits a \emph{Weierstrass normal form}, given by the very ample divisor $\fq^3$.
Indeed, suppose first that $p=2$, i.e., $F|K$ satisfies Theorem~\ref{2023_05_26_02:00}.
Substituting $x$ for $x+r$, where $r$ is the residue class $x(\fq) \in K$, we obtain $x(\fq)=0$, and 
after subtracting $z(\fq) \in K$ from $z$ we normalize $z(\fq) = 0$, i.e., $a_0=0$.
Defining $\breve x = x^{-1}$, $\breve z = z x^{-2}$ (for~\ref{2023_05_26_02:01}) or $\breve z=z x^{-1}$, $\breve w = z^2 x^{-1}$ (for~\ref{2023_05_26_02:02} and~\ref{2023_05_26_02:03}) we obtain the Weierstrass normal forms
\begin{enumerate}[(i)]
    \item $\breve z^2 = \breve x^3 + a_2 \breve x^2 + a_4$, where $a_2 \in K$, $a_4 \in K \setminus K^2$;
    \item $\breve z^2 = \breve w^3 + a_2 \breve w + b_2$, where $a_2 \in K \setminus K^2$, $b_2 \in K \setminus K^2(a_2)$;
    \item $\breve z^2 = \breve w^3 + a_2 \breve w$, where $a_2 \in K \setminus K^2$.
\end{enumerate}
If $p=3$, i.e., $F|K$ satisfies Theorem~\ref{2024_04_16_15:20}, we can normalize $z(\fq)=x(\fq)=a_0=0$, and by setting $\breve z = zx^{-1}$, $\breve x = x^{-1}$ we get the Weierstrass normal form
\[ \breve z^3 = \breve x^2 + a_3,  \quad \text{where $a_3 \in K \setminus K^3$}.\]
\end{rem}

Recall that the \emph{degree of imperfection} of a field $K$ of characteristic $p>0$ is the number of elements of any $p$-basis of the purely inseparable extension $K^p \su K$;
see \cite[V, p.\,170, Ex.\,1]{Bour90}.
Equivalently, this integer is equal to $\log_p [K:K^p]$.

\begin{prop}
Let $F|K$ be a quasi-elliptic function field. If the base field $K$ has degree of imperfection $1$, then the singular prime $\pp$ has degree $\deg(\pp) = p = \mathrm{char}(K)$.
\end{prop}

\begin{proof}
The case of item~\ref{2023_05_26_02:02} in Theorem~\ref{2023_05_26_02:00} cannot occur, since $K = K^2(a_2)$.
\end{proof}

Let $k$ be an algebraically closed field of characteristic $p\in \{ 2, 3 \}$.
Let $S$ be a quasi-elliptic surface, 
with associated quasi-elliptic fibration $f:S\to B$.
This means that the generic fibre of $f$ is a quasi-elliptic curve over $K=k(B)$, or equivalently, that each special fibre is a cuspidal rational curve.
As $B$ is a curve the field $K$ has degree of imperfection $1$.
Hence the proposition implies that the non-smooth locus of $f$, which is a smooth curve on $S$, is a purely inseparable cover of degree $p$ of $B$. In particular, it intersects each special fibre at a unique point, the cusp, with multiplicity $p$.
This was originally proved by Bombieri and Mumford \cite[Proposition~1]{BM76}; see also \cite[Proposition~4.1.14]{CDL23}.

\section{A singular prime that is a canonical divisor}
\label{2023_04_10_16:30}

Let $F|K$ be a one-dimensional separable function field of genus $g=3$ in characteristic $p=2$.
Suppose it is geometrically rational, i.e., the extended function field $\ov K F|\ov K$ has genus $\ov g=0$.
According to \cite[Corollary~2.7]{HiSt22}, the first and second Frobenius pullbacks $F_1|K$ and $F_2|K$ of $F|K$ have genera $g_1\leq 1$ and $g_2=0$.
If $g_1=0$ then the function field $F|K$ is hyperelliptic, because then $F_1|K$ is a quadratic subfield of genus zero.
Therefore, as our interest in this paper lies in non-hyperelliptic function fields, we assume in this section that $g_1=1$, which means that $F_1|K$ is a quasi-elliptic function field, i.e., there is a unique prime $\pp$ of $F|K$ whose restriction $\pp_1$ to $F_1|K$ is singular, with geometric singularity degree $\de(\pp_1)=1$.
Since $\delta(\pp_1)-\delta(\pp_2) \leq \frac 12 (\delta(\pp)-\delta(\pp_1))$ by \cite[Proposition~2.4]{HiSt22}, 
it follows that $\delta(\pp) \geq 3$, and therefore,
by Rosenlicht's genus drop formula, $\delta(\pp)=3$ and $\pp$ is the only singular prime of $F|K$.

In this section we investigate the class of function fields $F|K$ satisfying the above properties, 
plus the additional condition that the divisor $\pp$ is canonical.
Our analysis builds on Section~\ref{2023_06_20_17:00}, which supplies a full description of the quasi-elliptic Frobenius pullbacks $F_1|K$.
As follows from Theorem~\ref{2023_05_26_02:00} (or \cite[Theorem~2.24]{HiSt22}), 
the restricted prime $\pp_3$ is always rational.
However, the non-singular prime $\pp_2$ can be rational or non-rational.

\begin{thm}\label{2023_05_18_23:10}
A one-dimensional separable function field $F|K$ of characteristic $p=2$ and genus $g = 3$ is geometrically rational 
and admits a prime $\pp$ such that $\de(\pp) = 3$, $\de(\pp_1) = 1$, 
and such that $\pp$ is a canonical divisor,
if and only if
$F|K$ can be put into one of the following normal forms.

\begin{enumerate}[\upshape (i)]
    \item \label{2024_04_12_11:10} 
    $ F|K = K(x,y)|K$, where 
    \[ y^4 = a_0 + x + a_2 x^2 + a_4 x^4, \]
    and $a_0,a_2,a_4\in K$ are constants satisfying $a_4\notin K^2$.
    Here the singular prime $\pp$ has degrees $\deg(\pp)=4$, $\deg(\pp_1)=2$, $\deg(\pp_2)=1$, and residue fields $\ka({\pp}) = K(a_4^{1/4})$, $\ka({\pp_1}) = K(a_4^{1/2})$, $\ka(\pp_2) = K$.
    
    \item \label{2024_04_12_11:15}
    $F|K = K(x,z,y)|K$, where 
    \begin{align*}
        z^4 = a_0 + x + a_2 x^2, \quad
        y^2 = c_0 + c_1 x + z + c_2 z^2,
    \end{align*}
    and $a_0,a_2,c_0,c_1,c_2 \in K$ are constants satisfying $a_2\notin K^2$ and
    \begin{equation}\label{2024_04_12_14:45}
        c_1\neq 0 \, \text{ or }\, c_2\notin K^2 + K^2 a_2.
    \end{equation}
    Here $\pp$ has degrees $\deg(\pp)=4$, $\deg(\pp_1)=\deg(\pp_2)=2$, $\deg(\pp_3)=1$, and residue fields $\ka(\pp) = K(a_2^{1/2}, (c_2 + a_2^{-1/2} c_1)^{1/2})$, $\ka({\pp_1}) = \ka({\pp_2}) = K(a_2^{1/2})$, $\ka(\pp_3) = K$.
\end{enumerate}
In both cases the singular prime $\pp$ is the only pole of the function $x$.
\end{thm}

Note that condition~\eqref{2024_04_12_14:45} rephrases as $\ka(\pp)\supsetneqq \ka(\pp_1)$, i.e., $\pp$ is unramified over $F_1$.

\begin{proof}
Let $F|K$ be a function field as in the statement of the theorem and let $\pp$ be its only singular prime.
Recall that the condition that $\pp$ is a canonical divisor means that 
\[ \deg(\pp)=2g-2=4 \quad \text{ and } \quad \dim H^0(\pp)=g=3. \]
Recall also that the Frobenius pullback $F_1|K$ is quasi-elliptic, and that the restricted prime $\pp_3$ is rational.

We suppose first that the non-singular prime $\pp_2$ is rational.
By Theorem~\ref{2023_05_26_02:00}~\ref{2023_05_26_02:01}, the quasi-elliptic Frobenius pullback $F_1|K$ admits the following normal form
\[ F_1|K=K(x,z)|K, \, \text{ where } \, z^2 = a_0 + x + a_2 x^2 + a_4 x^4, \, a_i \in K,a_4\notin K^2, \]
and $\ka(\pp_1) = K(a_4^{1/2})$, $\ka(\pp_2) = K$.
By Remark~\ref{2023_12_21_22:35} we also have
\begin{align*}
    H^0(\pp_1) & 
       = K \oplus K x \quad \text{and} \quad
    H^0(\pp_1^2) = K \oplus K x \oplus K x^2 \oplus K z.
\end{align*}
Clearly $\deg(\pp) = 2 \deg(\pp_1)$, i.e., the extension $\pp|\pp_1$ is unramified.
Since $\dim H^0(\pp) = 3 > \dim H^0(\pp_1)=2$ and $H^0(\pp) \cap F_1 =  H^0(\pp_1)$, there exists a function $y\in F$ such that
\[ H^0(\pp) = H^0(\pp_1)\oplus Ky = K \oplus Kx \oplus Ky, \]
which is a separating variable of $F|K$, i.e., $F=F_1(y)=K(x,z,y)$.
Since the square $y^2$ lies in $H^0(\pp^2) \cap F_1 = H^0(\pp_1^2)$, but not in $F_2=K(x)$ as it is a separating variable of $F_1|K$, there are constants $b_i \in K$ with $b_3 \neq 0$ such that
\[ y^2 = b_0 + b_1 x  + b_2 x^2 + b_3 z. \]
Substituting $x$ and $z$ with $b_3^{-2} x$ and $b_3^{-1} z$ respectively we may normalize $b_3=1$. 
Replacing $z$ with $z + b_0 + b_1 x + b_2 x^2$ we can further normalize $b_0 = b_1 = b_2 = 0$, i.e., $z = y^2$.
This gives the normal form in~\ref{2024_04_12_11:10}.

Conversely, the normal form ensures that the function field $F|K=K(x,y)|K$ has the desired properties.
Indeed, 
the pole $\pp$ of $x$ is inertial over $F_2=K(x)$ with residue field $\ka(\pp) = K(a_4^{1/4})$, and 
has geometric singularity degree $\de(\pp)=2\de(\pp_1) + \frac12 v_{\pp_2} \big( d(\frac yx)^4 \big) = 3$ by \cite[Theorem~2.3]{BedSt87}.
By the Jacobian criterion there are no singular primes other than $\pp$, i.e., $F|K$ has genus $g=3$.
Since the divisor $\pp$ has degree $\deg(\pp) = 4 = 2g-2$, it follows from the Riemann-Roch theorem that $\dim H^0(\pp) \leq g = 3$, with equality if and only if the divisor $\pp$ is canonical.
As the power $y^2$  belongs to $ H^0(\pp_1^2) = H^0(\pp^2) \cap F_1$, so that $y\in H^0(\pp)$, we conclude that  $1,x,y$ are three linearly independent elements of $H^0(\pp)$.

\medskip

Now we assume that the restricted prime $\pp_2$ is non-rational.
Since the prime $\pp_3$ is rational, this means that the extension $\pp_2|\pp_3$ is unramified.
By Lemma~\ref{2024_04_12_11:25} below, the extension $\pp|\pp_1$ is also unramified.
Let $e_1$ denote the ramification index of $\pp_1|\pp_2$.
Since $\deg(\pp)=2\deg(\pp_1)$, the condition $\deg(\pp)=4$ means that $\deg(\pp_1)=\deg(\pp_2)=2$, i.e., $e_1=2$.
Then by Theorem~\ref{2023_05_26_02:00}~\ref{2023_05_26_02:03} the quasi-elliptic function field $F_1|K$ admits the following normal form
\[ F_1|K=K(x,z)|K, \,\text{ where }\, z^4 = a_0 + x + a_2 x^2, \, a_0 \in K, \, a_2 \in K \setminus K^2, \]
and $\ka(\pp_1) = \ka(\pp_2) = K(a_2^{1/2})$, $\ka(\pp_3)=K$.
In addition, we know from Remark~\ref{2023_12_21_22:35} that
\begin{align*}
    H^0(\pp_1) &= K \oplus Kz \quad\text{and} \quad
    H^0(\pp_1^2) = K \oplus Kx \oplus K z^2 \oplus Kz.
\end{align*}
Because $\dim H^0(\pp) = 3 > \dim H^0(\pp_1) = 2$ and $H^0(\pp) \cap F_1 = H^0(\pp_1)$ we can write
\[ H^0(\pp) = H^0(\pp_1) \oplus Ky = K \oplus Kz \oplus Ky, \]
where $y\in F$ is a separating variable of $F|K$, i.e., $F = F_1(y) = K(x,z,y)$.
Since $y^2$ lies in 
$H^0(\pp^{2}) \cap F_1 = H^0(\pp_1^{2})$, 
but not in $F_2 = K(x,z^2)$ as it is a separating variable of $F_1|K$, 
there exist constants $c_i\in K$ with $c_3\neq0$ such that
\[ y^2 = c_0 + c_1 x + c_2 z^2 + c_3 z. \]
Substituting $x, z$ with $c_3^{-4}x, c_3^{-1} z$ we can normalize $c_3=1$, thus obtaining the equations in the normal form in~\ref{2024_04_12_11:15}.
As follows from the equality $\ka(\pp) = K(a_2^{1/2}, (c_2 + a_2^{-1/2} c_1)^{1/2})$, which will be proved in the next paragraph, the condition that $\pp|\pp_1$ is unramified means that $c_1\neq0$ or $c_2\notin K^2 + K^2 a_2$.
 
It remains to verify 
that the normal form ensures the following:
that the function field $F|K = K(x,z,y)|K$ has genus $3$, 
that the pole $\pp$ of $x$ has geometric singularity degree $3$, 
and that the divisor $\pp$ is canonical.
Consider
the function $\cx := x^{-1}\in F_3$, 
a local parameter at the rational prime $\pp_3$, and the functions $ \breve{z}:= z x^{-1} \in F_1$ and $\cy:= yx^{-1} \in F$, which satisfy the relations
\[ 
\breve{z}^4 = a_2 \cx^2 + \cx^3 + a_0 \cx^4 \quad \text{and} \quad
\cy^2 = c_1 \cx + c_0 \cx^2 + c_2 \breve{z}^2 + \cx \breve{z}. 
\]
Computation shows that $\breve{y}^8$ and $\breve{z}^8$ are polynomial expressions in $\breve x$ so that 
\[ \bigg( \frac{\cy}{\breve{z}}\bigg)^8 
= (c_2^2 + a_2^{-1} c_1^2)^2 + (a_2^{-1} + a_2^{-4} c_1^4) \cx^2 + a_2^{-2} \cx^3 + \cdots.
\]
If $\frac{\cy}{\breve{z}} (\pp) \notin \ka({\pp_1})	= K(a_2^{1/2})$, then $\pp$ is inertial over $F_1$ and $\de(\pp) = 2 \cdot 1 + \frac 12 v_{\pp_3} \big( d( \frac{\cy}{\breve{z}} )^8 \big) = 3$ by \cite[Theorem~2.3]{BedSt87}. 
In the opposite case $\frac{\cy}{\breve{z}} (\pp) \in K\big(\frac{\breve{z}^2}{\cx}(\pp)\big)$, say $t(\pp) = 0$ for some $t$ in $\frac{\cy}{\breve{z}}	+ K + K {\cdot} \frac{\breve{z}^2}{\cx}$, the prime $\pp$ is ramified over $F_1$ with local parameter $t$ because
\[ t^8 = (a_2^{-1} + a_2^{-4} c_1^4) \cx^2 + a_2^{-2} \cx^3 + \cdots,
\]
and therefore $\de(\pp) = 2\cdot 1 + \frac12 v_{\pp_3}(dt^8) = 3$. 
Thus 
$\de(\pp)=3$.
Since 
$y^8 = f(x)$ with $f'(x)=1$,
we conclude from the Jacobian criterion that there are no singular primes other than $\pp$, whence $F|K$ has genus $g=3$.

We finally check that the divisor $\pp$ is canonical. Clearly $\deg(\pp) = 2\deg(\pp_1) = 4 = 2g-2$. As the squares $z^2$ and $y^2$ both lie in $H^0(\pp_1^2) = H^0(\pp^{2}) \cap F_1$, hence  $K \oplus Kz \oplus Ky\subseteq H^0(\pp)$, we conclude that $\dim H^0(\pp) = 3 = g$, i.e., $\pp$ is a canonical divisor.
\end{proof}

In the above proof we used the following technical result, which is valid in every characteristic $p>0$.

\begin{lem}\label{2024_04_12_11:25}
Let $F|K$ be a one-dimensional separable function field of characteristic $p>0$.
Let $\pp$ be a singular prime that is ramified over $F_1$, i.e., $\ka(\pp)=\ka(\pp_1)$.
Then 
\[ \de(\pp) \geq p \, \de(\pp_1) + \tfrac{p-1}2 \cdot \deg(\pp). \]
\end{lem}

Note that as $\ka(\pp)=\ka(\pp_1)$ we can replace $\deg(\pp)$ with $\deg(\pp_1)$ in the above inequality.

\begin{proof}
Assume first that $\pp$ is non-decomposed, i.e., the extension $\ka(\pp)|K$ is purely inseparable, i.e., some restriction $\pp_n$ is rational \cite[Corollary~2.17]{HiSt22}.
Then $\pp$ is non-rational and $\deg(\pp)>1$ is a $p$-power.
Choose an integer $n>1$ such that $\pp_n$ is rational and let $z\in F$ be a local parameter at $\pp$.
Since the function $z^{p^n} \in F_n$ has order $p^n/e_0\cdots e_{n-1} = \deg(\pp)$ at $\pp_n$, where $e_i$ denotes the ramification index of the extension $\pp_i|\pp_{i+1}$,
the order of the differential $dz^{p^n}$ of $F_n|K$ at $\pp_n$ is at least $\deg(\pp)$. Thus the claim follows from \cite[Theorem~2.3]{BedSt87}.

Now we drop the non-decomposedness assumption on $\pp$.
Let $\fq$ be a (non-decomposed) prime of the extended function field $K^{sep} F | K^{sep}$ that lies over $\pp$, where $K^{sep}$ denotes the separable closure of $K$.
By \cite[Proposition~2.12]{HiSt22} we have $\de(\pp) = r \de(\fq)$, $\de(\pp_1) = r \de(\fq_1)$ and  $\deg(\pp) = r \deg(\fq)$, where $r$ is the separability degree of the extension $\ka(\pp)|K$.
\end{proof}

\begin{rem}\label{2024_04_20_15:30}
A function field $F|K$ admitting a singular prime $\pp$ whose degree is a power of 2 can only exist in characteristic $p=2$.
Indeed, as the residue field extension $\ka(\pp)|K$ is inseparable (see e.g.~\cite[Proposition~2.9]{HiSt22}), its degree $\deg(\pp)$ must be a multiple of the characteristic $p$.
In particular, the assumption that $p=2$ in Theorem~\ref{2023_05_18_23:10} can be removed.
\end{rem}

\begin{rem}\label{2023_11_07_14:00}
We draw some consequences from the proof of the theorem.
For a function field $F|K = K(x,y)|K$ as in Theorem~\ref{2023_05_18_23:10}~\ref{2024_04_12_11:10} we have $e=e_1=1$ and
$ \divv_\infty(x) = \pp, $
where $e$ and $e_1$ denote the ramification indices of $\pp$ and $\pp_1$ over $F_1$ and $F_2$ respectively.
In addition the Frobenius pullback
$F_1|K = K(x,y^2)|K$
is quasi-elliptic, actually given as in Theorem~\ref{2023_05_26_02:00}~\ref{2023_05_26_02:01},
and moreover
\begin{align*}
    H^0(\pp_2) &= H^0(\pp_1) = K \oplus Kx, \quad
    H^0(\pp) = K \oplus Kx \oplus Ky.
\end{align*}
Similarly, for a function field $F|K=K(x,z,y)|K$ as in Theorem~\ref{2023_05_18_23:10}~\ref{2024_04_12_11:15}, it is clear that
$e=1$, $e_1=2$, $e_2 = 1$, and
$ \divv_\infty(x) = \pp^{2} $,
where $e$, $e_1$ and $e_2$ are the ramification indices of $\pp$, $\pp_1$ and $\pp_2$ over $F_1$, $F_2$ and $F_3$ respectively.
Furthermore, the quasi-elliptic Frobenius pullback $F_1|K = K(x,z)|K$ satisfies Theorem~\ref{2023_05_26_02:00}~\ref{2023_05_26_02:03}, and 
\begin{align*}
    H^0(\pp_3) &= K \oplus Kx, \quad
    H^0(\pp_1) = K \oplus Kz, \quad
    H^0(\pp) = K \oplus Kz \oplus Ky. 
\end{align*}
\end{rem}

By proceeding as in the proof of Proposition~\ref{2024_08_10_15:40}, we can use the above Riemann-Roch spaces together with the incidence properties of the functions $x,z,y$ to determine the isomorphism classes of the function fields in Theorem~\ref{2023_05_18_23:10}.

\begin{prop}\label{2024_08_11_01:30}
$ $
\begin{enumerate}[\upshape (i)]
    \item \label{2024_08_11_01:31}
    Let $F|K$ and $F'|K$ be two function fields of type~\ref{2024_04_12_11:10} in Theorem~\ref{2023_05_18_23:10}.
    Then the $K$-isomorphisms $F'\overset\sim\to F$ are given by the transformations
    \[ (x',y') \mapsto ( \eps^4 ( x + \al ), \eps (y + \ga x + \beta) ), \]
    where $\eps,\al,\beta,\ga \in K$ are constants such that 
    $\eps\neq 0$,
    $\eps^{4} a_2' = a_2$, 
    $\eps^{12} a_4' = a_4 + \ga^4$,
    $\eps^{-4} a_0' = a_0 + \al^2 a_2 + \al^4 a_4 + \beta^4 + \ga^4 \al^4 + \al$.
    The automorphisms of $F|K$ are given by
    \[ (x,y) \mapsto ( \eps ( x + \al ), \eps (y + \beta) ), \]
    where $\eps,\al,\beta \in K$ satisfy
    $\eps^3 = 1$,
    $(\eps + 1) a_2 = 0$,
    $\beta^4 = (\eps^{2} + 1) a_0 + \al^2 a_2 + \al^4 a_4 + \al$.
    
    \item \label{2024_08_11_01:32}
    Let $F|K$ and $F'|K$ be two function fields of type~\ref{2024_04_12_11:15} in Theorem~\ref{2023_05_18_23:10}.
    Then the $K$-isomorphisms $F'\overset\sim\to F$ are given by the transformations
    \[ (x',z',y') \mapsto ( \eps^8 ( x + \al ), \eps^2 (z + \beta), \eps (y + \tau + \ga z) ), \]
    where $\eps,\al,\beta,\tau,\ga \in K$ are constants such that
    $\eps\neq 0$,
    $\eps^8 a_2' = a_2$,
    $\eps^6 c_1' =  c_1$,
    $\eps^2 c_2' = c_2 + \ga^2$,
    $\eps^{-8} a_0' = a_0 + \al^2 a_2 + \al + \beta^4$,
    $\eps^{-2} c_0' = c_0 + \al c_1 + \beta^2 (c_2 + \ga^2) + \beta + \tau^2$.
    In particular the quotient $\iota = c_1^4/a_2^3$ is an invariant of the function field $F|K$.
    The automorphisms of $F|K$ are given by
    \[ (x,z,y) \mapsto (x + \al, z + \beta, y + \tau), \]
    where $\al,\beta,\tau \in K$ satisfy
    $\beta^4 = \al + \al^2 a_2 $,
    $\tau^2 = \al c_1 + \beta + \beta^2 c_2 $.
\end{enumerate}
\end{prop}

\section{Non-hyperelliptic regular non-smooth curves}
\label{2023_05_24_02:00}

Our next objective is to obtain the projective regular models of the function fields in the preceding section.
We are particularly interested in those function fields which are non-hyperelliptic.
Recall that the \textit{canonical field} of a one-dimensional function field $F|K$ is the subfield generated by the quotients of the non-zero holomorphic differentials of $F|K$.
Equivalently, the canonical field of $F|K$ is the subfield generated over $K$ by the global sections of any canonical effective divisor.
A function field $F|K$ of genus $g\geq2$ is called \textit{hyperelliptic} if it admits a quadratic subfield of genus zero;
this subfield is equal to the canonical field and so it is uniquely determined.

\begin{prop}\label{2024_04_16_19:10}
Every function field $F|K$ in Theorem~\ref{2023_05_18_23:10}~\ref{2024_04_12_11:10} is non-hyperelliptic.
\end{prop}

\begin{proof}
The canonical field is generated over $K$ by $H^0(\pp) = K \oplus Kx \oplus Ky$, as follows from Remark~\ref{2023_11_07_14:00}.
\end{proof}

\begin{prop}\label{2022_03_08_13:10}
Let $F|K$ be a function field as in Theorem~\ref{2023_05_18_23:10}~\ref{2024_04_12_11:15}. Then its canonical field is equal to $K(z,y)$.
The function field $F|K$ is non-hyperelliptic if and only if $c_1 \neq 0$, i.e., if and only if the invariant $\iota = c_1^4/a_2^3$ does not vanish.
\end{prop}

In particular, there are infinitely many isomorphism classes of non-hyperelliptic function fields in Theorem~\ref{2023_05_18_23:10}.

\begin{proof}
The first assertion follows from
$H^0(\pp) = K \oplus Kz \oplus Ky$.
To prove the second assertion we first assume that $c_1=0$.
Then the canonical field $H:=K(z,y)$ has genus zero because its generators satisfy the quadratic equation
\[ y^2 + c_2 z^2 + z + c_0 = 0. \]
In particular, $H$ is a proper subfield of $F = K(x,z,y) = H(x)$.
The element $x$ satisfies over the canonical field $H$ the quadratic 
equation
\[ a_2 x^2 + x + a_0 + z^4 = 0. \]
Now we assume that $c_1\neq 0$.
Then the canonical field $K(z,y)$ coincides with $F$ because it contains the function $x$.
\end{proof}

If a function field $F|K$ of genus $g=3$ is non-hyperelliptic then the canonical linear system of divisors realizes the regular proper model of $F|K$ as a (canonical) quartic curve in $\PP^{g-1}(K) = \PP^2(K)$.
For a non-hyperelliptic function field as in Theorem~\ref{2023_05_18_23:10}, the complete linear system $|\pp|$ provides such a realization.

\begin{thm}\label{2023_04_10_16:50}
The one-dimensional separable non-hyperelliptic geometrically rational function fields $F|K$ of genus $g=3$ admitting a 
singular prime $\pp$ that is a canonical divisor have characteristic $p=2$ and are classified as follows:
\begin{enumerate}[\upshape (i)]
   \item \label{2023_05_19_18:01}
   If $\pp_2$ is rational then $F|K$ is the function field of a regular plane projective integral quartic curve over $\Spec K$ with generic point $(x:y:1)$ that satisfies the quartic equation
\[ y^4 + a + x + b x^2 + c x^4 = 0, \]
where $a,b,c\in K$ are constants satisfying $c\notin K^2$.
The singular prime $\pp$ is centered at the point $(1 : c^{1/4} : 0)$ in $\PP^2(\ov K)$ and has residue fields $\ka(\pp) = K(c^{1/4})$, $\ka(\pp_1) = K(c^{1/2})$, $\ka(\pp_2) = K$.
   \item \label{2023_05_19_18:02}
   If $\pp_2$ is non-rational then $F|K$ is the function field of a regular plane projective integral quartic curve over $\Spec K$ with generic point $(1:y:z)$ that satisfies the quartic equation
\[ y^4 + a z^4 + b y^2 + c z^2 + b z + d = 0, \]
where $a,b,c,d \in K$ are constants satisfying $a \notin K^2$ and $b \neq 0$.
The singular prime $\pp$ is centered at the point $(0 : a^{1/4} : 1)$ in $\PP^2(\ov K)$ and has residue fields $\ka(\pp) = K(a^{1/4})$, $\ka({\pp_1}) = \ka({\pp_2}) = K(a^{1/2})$, $\ka(\pp_3) = K$.
Moreover, the polynomial expression $\iota = ab^2 + c^2 + 1$ is an invariant of the function field $F|K$.
\end{enumerate}
\end{thm}

\begin{proof}
The fact that $p=2$ follows from Remark~\ref{2024_04_20_15:30}, since $\deg(\pp)=2g-2=4$.
As remarked in the beginning of Section~\ref{2023_04_10_16:30}, the assumption that $F|K$ is non-hyperelliptic implies that $\de(\pp)=3$ and $\de(\pp_1)=1$.
Consequently, the function fields we want to classify are necessarily given as in Theorem~\ref{2023_05_18_23:10}.

(i) Conversely, each function field in Theorem~\ref{2023_05_18_23:10}~\ref{2024_04_12_11:10} is non-hyperelliptic by Proposition~\ref{2024_04_16_19:10}, and so it only remains to simplify the notation by setting 
\[ a:= a_0, \quad b:= a_2, \quad c:= a_4. \]

(ii)
As by Proposition~\ref{2022_03_08_13:10} a function field $F|K$ in Theorem~\ref{2023_05_18_23:10}~\ref{2024_04_12_11:15} is non-hyperelliptic if and only if $c_1 \neq 0$, we can normalize $c_0 = 0$ by replacing $x$ with $x + c_1^{-1} c_0$.
Eliminating the variable $x$ via the equation $y^2=c_1 x + z + c_2 z^2$ we obtain the relation
\[ a_2 y^4 + (c_1^2 + a_2 c_2^2) z^4 + c_1 y^2 + (a_2 + c_1 c_2) z^2 + c_1 z + c_1^2 a_0 = 0. \]
By setting
\[ 
a:= a_2^{-1} (c_1^2 + a_2 c_2^2),   \quad
b:= a_2^{-1} c_1 , \quad
c := a_2^{-1} (a_2 + c_1 c_2), \quad
d := a_2^{-1} c_1^2 a_0,
\]
we get a bijection $(a_2,c_1,a_0,c_2)\mapsto(a,b,c,d)$ from the product set $(K\setminus K^2)\times K^* \times K\times K$ onto itself, 
whose inverse is given by
\[ a_2 = b^{-4} \iota, \quad c_1 = b^{-3} \iota, \quad a_0 = b^2 d \, \iota^{-1}, \quad c_2 = b^{-1}(1+c), \]
where $\iota = ab^2 + c^2 + 1 = a_2^{-3} c_1^4$ is the invariant of $F|K$ (see Proposition~\ref{2024_08_11_01:30}).

In both cases the point in $\PP^2(\ov K)$ at which the singular prime $\pp$ is centered is obtained from the Jacobian criterion.
The assertions on the residue fields of $\pp$ follow from the description of the residue fields in Theorem~\ref{2023_05_18_23:10}.
\end{proof}

Recall that in view of Theorem~\ref{2023_05_18_23:10} the iterated Frobenius pullback $F_2|K = F^4{\cdot}K|K$ (resp. $F_3|K=F^8{\cdot}K|K$) is a rational function field if $F|K$ is of type~\ref{2023_05_19_18:01} (resp. type~\ref{2023_05_19_18:02}).

\begin{cor}\label{2024_06_19_15:30} $ $
\begin{enumerate}[\upshape (i)]
    \item \label{2024_06_19_15:31}
     If $F|K$ is of type~\ref{2023_05_19_18:01} then the second Frobenius pullback $F_2|K$ is the function field of the rational quartic curve over $\Spec K$ with the generic point $(x^4:y^4:1)$, which admits the parameterization $(x^4 : a + x + bx^2 + cx^4:1)$.
    
    \item \label{2024_06_19_15:32}
    If $F|K$ is of type~\ref{2023_05_19_18:02} then the third Frobenius pullback $F_3|K$ is the function field of the rational quartic curve over $\Spec K$ with the generic point $(1:y^8:z^8)$, which can be parameterized as follows
    \begin{align*}
        z^8 &= b^4 d^2 \iota^{-2} + x^2 + b^{-8} \iota^2 x^4, \\
        y^8 &= \iota^{-2} d (b^2 \iota + d + c^4 d) + x + b^{-4} (\iota +1 + c^4) x^2 + b^{-8} a^2 \iota^2 x^4.
    \end{align*}
\end{enumerate}
\end{cor}

The isomorphism classes are readily obtained from items~\ref{2024_08_11_01:31} and~\ref{2024_08_11_01:32} in Proposition~\ref{2024_08_11_01:30}.
In item~\ref{2024_08_11_01:32} the conditions have to be rephrased in terms of the constants $a,b,c,d$.
In item~\ref{2024_08_11_01:31} no change is necessary (just write $a,b,c$ for $a_0,a_2,a_4$), so we do not repeat it.



\begin{prop}\label{2023_06_05_16:35}
$ $
\begin{enumerate}[\upshape (i)]
    \setcounter{enumi}{1}
    \item
    
    Let $F|K=K(z,y)|K$ and $F'|K=K(z',y')|K$ be two function fields of type~\ref{2023_05_19_18:02} in Theorem~\ref{2023_04_10_16:50}.
    Then the $K$-isomorphisms $F' \overset \sim \to F$ are given by
    \[ (z',y') \mapsto ( \eps^2 (z + \beta), \eps (y + \tau + \ga z) ), \]
    where $\eps,\beta,\tau,\ga \in K$ satisfy 
    $\eps \neq 0$, 
    $\eps^4 a' = a + \ga^4$,
    $\eps^{-2} b' = b$,
    $c' = c + \ga^2 b$,
    $\eps^{-4} d' = d + \beta^4 (\ga^4 + a) + \beta^2 c + \tau^4 + ( \beta^2 \ga^2 + \beta + \tau^2) b$. 
    The automorphisms of $F|K$ are given by
    \[ (z,y) \mapsto (z + \beta, y + \tau), \]
    where $\beta,\tau \in K$ satisfy 
    $\beta^4 a + \beta^2 c + \tau^4 + ( \beta + \tau^2) b = 0$.
\end{enumerate}
\end{prop}

The proposition shows again that $\iota = a b^2 + c^2 + 1$ is an invariant of the function field $F|K$ of type~\ref{2023_05_19_18:02}.

\begin{cor}\label{2024_07_01_18:25}
Let $F|K$ be a function field as in Theorem~\ref{2023_04_10_16:50}.
Then we can normalize
\begin{enumerate}[\upshape (i)]
    \item $a=0$ in item~\ref{2023_05_19_18:01},
    \item $d=0$ in item~\ref{2023_05_19_18:02},
\end{enumerate}
after extending the base field $K$ via a separable extension of degree $\leq 4$.
\end{cor}

Let $F|K$ be a function field as in Theorem~\ref{2023_04_10_16:50}.
Then it is the function field of a plane projective geometrically integral quartic curve $C$ over $\Spec K$, as described in the two items~\ref{2023_05_19_18:01} and~\ref{2023_05_19_18:02}.
The extended quartic curve $C\otimes_K \ov K$ in the projective plane $\PP^2(\ov K)$ is rational and has a unique singular point, which by the Jacobian criterion is equal to $(1:c^{1/4}:0)$ or $(0 : a^{1/4} : 1)$
respectively.
This is the point where the only singular prime $\pp$ is centered.
Its local ring $\OO_\pp \otimes_K \ov K$ has geometric singularity degree three and its tangent line does not intersect the quartic curve at any other point.

The quartic curve $C|K$ is the regular complete model of the function field $F|K$; more precisely, its closed points correspond bijectively to the primes of $F|K$, and its local rings are the local rings of the corresponding primes (which are discrete valuation rings, i.e., regular one-dimensional local rings).
The generic point of the curve $C|K$ is the only non-closed point, and its local ring is the function field $F|K$.

The extended quartic curve $C \otimes_K \ov K$ in the projective plane $\PP^2(\ov K)$ is \textit{strange}; indeed, its tangent lines have the common intersection point $(0:1:0)$.
If item~\ref{2023_05_19_18:01} occurs then
the tangent line at each non-singular point meets the quartic curve only at the tangency point;
in particular, every non-singular point is a non-ordinary inflection point.
By contrast, 
if item~\ref{2023_05_19_18:02} occurs then
each tangent line at a non-singular point is a bitangent,
and so the quartic curve has no inflection points.

\section{Universal fibrations by rational quartics in characteristic two}
\label{2023_09_21_20:25}

The function fields in Theorem~\ref{2023_04_10_16:50} give rise to fibrations by rational plane projective quartic curves in characteristic $2$, which we now investigate.
We start with the second item in the theorem.
Let $k$ be an algebraically closed ground field of characteristic $2$.
Consider the smooth integral fivefold
\begin{align*} 
    Z_2 &= \{\, ((x:y:z), (a,b,c,d)) \in \PP^2 \times \AAA^4 \mid 
    y^4 + a z^4 + b x^2 y^2 + c x^2 z^2 + b x^3 z + d x^4 = 0   \,\}
\end{align*}
inside $\PP^2 \times \AAA^4$.
The projection morphism
\[ \pi_2:Z_2 \tto \AAA^4 \]
is proper and flat, 
and it provides a $4$-dimensional family of plane projective quartic curves.

Let $B$ be a closed subvariety of $\AAA^4$, and let $a,b,c,d$ denote its affine coordinate functions.
By restricting the base of the fibration $\pi_2:Z_2\to \AAA^4$ to $B$ we obtain a closed subfibration
\[ T \tto B \]
with total space $T=\pi_2^{-1}(B)\subseteq Z \su \PP^2\times \AAA^4$.
Clearly $k(B)=k(a,b,c,d)$,
and then the closed subfibration $T \to B$ satisfies the two conditions $b\neq 0$ and $a \notin k(a^2,b^2,c^2,d^2)$ in Theorem~\ref{2023_04_10_16:50}~\ref{2023_05_19_18:02} if and only if $B$ is not contained in the hyperplane $\{b=0\}$ and the tangent spaces of $B$ are not all contained in the hyperplanes $\{ a = \text{constant}\}$.
Note that the first condition $b \neq 0$ is preserved under dominant base extensions, while the second condition $a \notin k(B)^2$ is preserved under dominant separable base extensions.
Assuming both conditions are satisfied, if we dehomogenize $x\mapsto 1$ then the field of rational functions on the total space $T$ becomes
\[ k(T) = k(B) (z,y), \,\,\, \text{where} \,\,\, y^4 + a z^4 + b y^2 + c z^2 + b z + d = 0, \]
and $a,b,c,d \in k(B)$, $a\notin k(B)^2$ and $b\neq 0$.
If we replace the ground field $k$ by the function field $k(B)$ of the base, then the higher dimensional function field $k(T)|k$ of the total space becomes a one-dimensional function field $k(T)|k(B)$, which is just the function field of the generic fibre of the fibration $T\to B$ and has been characterized in Theorem~\ref{2023_04_10_16:50}~\ref{2023_05_19_18:02}.
When the base $B$ is the line defined by the equations $b=a$, $c=1$ and $d=0$, we obtain the pencil of quartics that has been studied in detail in \cite[Section~3]{HiSt22}.

By Corollary~\ref{2024_06_19_15:30}~\ref{2024_06_19_15:32} there is a dominant map $\psi: \AAA^1 \times B \ttto T$ given by
\[ (\bar x,(a,b,c,d)) \mapsto ( (1 : \bar y : \bar z) , (a,b,c,d)). \]
where $\bar y$ and $\bar z$ stand for the corresponding polynomial expressions in $\bar x^{1/8}$ whose coefficients are rational expressions in $a^{1/8}$, $b^{1/4}$, $c^{1/4}$ and $d^{1/8}$.
However $\psi$ is not a rational map.
To repair this we consider the third \emph{Frobenius pushforward}
\[ B' = \{ (a,b,c,d) \in \AAA^4 \mid f(a^8,b^8,c^8,d^8) = 0 \text{ for each $f\in I(B)$} \} \]
of the variety $B\subset \AAA^4$, and take the dominant rational map
\begin{align*}
    \AAA^1 \times B' &\ttto T \\
    (\bar x,(a,b,c,d)) &\mapsto \psi (\bar x^8,(a^8,b^8,c^8,d^8)).
\end{align*}
This shows in a very explicit way that the total space $T$ is uniruled (see \cite[Proposition~4.2]{HiSt24} for a more general proof).
However, as the dominant rational map is inseparable, this does not assure that the Kodaira dimension of $T$ is $-\infty$ (see \cite{Shi74} and \cite[p.\,265]{Lied13}).

\begin{thm}\label{2023_05_22_23:55}
Let 
$\phi:T\to B$ be a proper dominant morphism of irreducible smooth algebraic varieties whose generic fibre is a non-hyperelliptic geometrically rational curve $C$ over $k(B)$ of genus $3$, that admits a non-smooth point $\pp$ with non-rational image $\pp_2\in C_2$ and such that $\pp$ is a canonical divisor.
Then the fibration $\phi:T\to B$ is, up to birational equivalence, a dominant base extension of a closed subfibration of $\pi_2:Z_2\to \AAA^4$.
\end{thm}

\begin{proof}
By Theorem~\ref{2023_04_10_16:50}~\ref{2023_05_19_18:02}, the function field $k(T)|k(B)$ of the generic fibre can be put into the following normal form
\[ k(T) = k(B) (z,y), \,\,\, \text{where} \,\,\, y^4 + a z^4 + b y^2 + c z^2 + b z + d = 0, \]
and $a,b,c,d \in k(B)$, $a\notin k(B)^2$ and $b\neq 0$.
By restricting eventually the base $B$ to an open subset, 
we may assume that $B$ is affine and that $a,b,c,d$ are regular functions on $B$. By applying the birational transformation
\[   P \mapstto \big( (1:y(P):z(P)),\phi(P) \big)  \]
we can assume that 
\[ T = \{ ((\bar x:\bar y:\bar z),w) \in \PP^2 \times B \mid \bar y^4 + a(w) \bar z^4 + b(w) \bar x^2 \bar y^2 + c (w) \bar x^2 \bar z^2 + b(w) \bar x^3 \bar z + d(w) \bar x^4 = 0 \} \]
and that $\phi:T \to B$ is the second projection morphism. 
The regular functions $a,b,c,d$ define a dominant morphism $B\to B'$, where $B'$ is a closed subvariety of $\AAA^4$ whose coordinate algebra $k[B']$ is isomorphic to $k[a,b,c,d]$.
By restricting the base of the fibration $\pi_2:Z_2\to \AAA^4$ to $B'$ we obtain a closed subfibration
\[ T'\tto B' \]
with the total space $T' = \pi^{-1}(B') \subseteq Z \subseteq \PP^2 \times \AAA^4$.
Now, making a base change by the dominant morphism $B \to B'$ we obtain 
$   T = T' \times_{B'} B $.
\end{proof}

We describe the fibres of $\pi_2:Z_2\to \AAA^4$. 
If $b=0$ then the fibre over a point $P:=(a,b,c,d)$ in the base $\AAA^4$ is non-reduced, since it is a double smooth quadric or a quadruple line.

Assume next that $b \neq 0$. 
If 
$c\neq b a^{1/2}$, i.e.,
$\iota := a b^2 + c^2 + 1 \neq 1$,
then we are in the generic case, where the fibre $\pi_2^{-1}(P)$ has the same properties as the extended curve $C \otimes_K \ov K$ described in the previous section.
In other words, the fibre is an integral rational plane projective quartic with a unique singular point of multiplicity $2$, whose tangent line does not meet the curve at any other point; 
it is strange as all its tangent lines have the common point $(0:1:0)$,
and furthermore it has no inflection points, since every tangent line at a non-singular point is a bitangent.

If $\iota = 1$ then the fibre $\pi_2^{-1}(P)$ has the same properties as the fibres in the generic case, the only difference being that the multiplicity of the singular point is equal to $3$.

Now we build a fibration out of the function fields in Theorem~\ref{2023_04_10_16:50}~\ref{2023_05_19_18:01}.
The variety
\begin{align*} 
    Z_1 &= \{\, ((x:y:z), (a , b , c )) \in \PP^2 \times \AAA^3 \mid 
    y^4 + a z^4 + x z^3 + b x^2 z^2 + c x^4 = 0   \,\}
\end{align*}
is a smooth integral fourfold whose projection
\[ \pi_1:Z_1 \tto \AAA^3 \]
is proper and flat.
If $b\neq0$ then the fibre over the point $(a,b,c)$ has the same properties as the geometric generic fibre described in the previous section. 
More precisely, in this case the fibre is an integral plane projective rational quartic curve with a unique singular point of multiplicity $2$.
Its non-singular points are non-ordinary inflection points and every tangent line cuts the curve exactly at one point.
In addition, all such tangent lines pass through the point $(0:1:0)$ and so the quartic curve is strange.

If $b=0$ then the fibre has the same properties as those in the generic case $b\neq0$, with just one difference: the only singular point of the quartic curve has multiplicity $3$.

\begin{rem*}
If by homogenizing we enlarge the base of $\pi_1$ from $\AAA^3$ to $\PP^3$ then the total space remains smooth and the resulting fibration stays proper and flat. 
Over each point of $\PP^3 \setminus \AAA^3$ the fibre degenerates to a non-reduced quartic curve, which can be either a quadruple line or the union of two double lines.
\end{rem*}

Arguing as in the proof of Theorem~\ref{2023_05_22_23:55}, we can show that the fibration $\pi_1$ is universal in the sense that any other fibration whose generic fibre satisfies the properties of Theorem~\ref{2023_04_10_16:50}~\ref{2023_05_19_18:01} is a base extension of $\pi_1$.

\begin{thm}\label{2023_05_23_16:35}
Let 
$\phi:T\to B$ be a proper dominant morphism of irreducible smooth algebraic varieties whose generic fibre is a non-hyperelliptic geometrically rational curve $C$ over $k(B)$ of genus $3$, that admits a non-smooth point $\pp$ with rational image $\pp_2\in C_2$ and such that $\pp$ is a canonical divisor.
Then the fibration $\phi:T\to B$ is, up to birational equivalence, a dominant base extension of a closed subfibration of $\pi_1: Z_1 \to \AAA^3$.
\end{thm}

\begin{rem}\label{2023_06_13_21:50}
Let $B$ be a closed subvariety of $\AAA^3$, and let $a,b,c$ be its affine coordinate functions.
Then the closed subfibration $\pi_1^{-1}(B) \to B$ of $Z_1 \to \AAA^3$ satisfies the condition $c \notin k(B)^2 = k(a^2,b^2,c^2)$ of Theorem~\ref{2023_04_10_16:50}~\ref{2023_05_19_18:01} if and only if the tangent spaces of $B$ are not all contained in the hyperplanes $\{ c = \text{constant}\}$.
This condition is clearly preserved under dominant separable base extensions.
Moreover the total space $\pi_1^{-1}(B)$ is uniruled.
The dimension of $B$ can be diminished by letting some of the functions $a,b$ take fixed values,
but this cannot be done with $c$ because the condition $c\notin k(B)^2$ must be satisfied.
\end{rem}

\section{A moving singularity of multiplicity three}
\label{2024_07_01_19:25}

Let $C$ be a curve over $K$ defined as in Theorem~\ref{2023_04_10_16:50}, with function field $F|K = K(C)|K$. Recall that the only singular prime $\pp$ of $F|K$ is centered at the only singular point on the extended quartic curve $C_{\ov K} = C \otimes_K \ov K$.

\begin{prop}
The only singular point on $C_{\ov K}$ has multiplicity two or three.
The case of multiplicity three occurs if and only if $C$ is given as in item~\ref{2023_05_19_18:01} in Theorem~\ref{2023_04_10_16:50} with $b=0$.
\end{prop}

The proposition follows from the explicit equations defining $C$.
Besides, Corollary~\ref{2024_07_01_18:25} shows that in the case of multiplicity three (item~\ref{2023_05_19_18:01} with $b=0$) it can be assumed that $a=0$ after an eventual separable base field extension.
The resulting function field gives rise to a pencil of plane projective rational quartic curves, which we discuss in this section.
We find its minimal regular model and we further describe the pencil as a purely inseparable double cover of a quasi-elliptic fibration.

\medskip

Let $k$ be an algebraically closed field of characteristic $2$.
Consider the projective algebraic surface over $k$
\[ S\subset \PP^2 \times \PP^1 \]
defined by the bihomogeneous polynomial equation
\[ T_0 (Y^4 + X Z^3) + T_1 X^4 = 0, \]
where $X,Y,Z$ and $T_0,T_1$ denote the homogeneous coordinates of $\PP^2$ and $\PP^1$ respectively.
By 
the Jacobian criterion, this surface has a unique singularity at \mbox{$P=((0:0:1),(0:1))$}.
The second projection
\[ \phi:S\tto \PP^1 \]
yields a proper flat fibration by plane projective curves over $\PP^1$.
A fibre
over a point of the form $(1:c)$ is isomorphic to the plane projective rational quartic given by 
\[ Y^4 + X Z^3 + c X^4 = 0, \]
which has a unique singular point at $(1:c^{1/4}:0)$.
This curve has arithmetic genus 3, by the genus-degree formula for plane curves, and it enjoys the same properties as the fibres of the fibration $\pi_1:Z_1\to \AAA^3$ in Section~\ref{2023_09_21_20:25} in the non-generic case. 
The fibre over the point $(0:1)$ degenerates to 
a non-reduced quadruple line, which we call the \emph{bad fibre} of the fibration.

The first projection
$ S\to \PP^2 $
is a birational morphism whose inverse 
\[ \PP^2 \ttto S, \quad (x:y:z) \mapsto ((x:y:z),(x^4 : y^4 + xz^3)) \]
is undefined only at $(0:0:1)$.
More precisely, the map $S\to \PP^2$ contracts the horizontal curve $(0:0:1) \times \PP^1$ to the point $(0:0:1)$ and restricts to an isomorphism
\[ S \setminus (0:0:1) \times \PP^1 \overset\sim\tto \PP^2 \setminus \{(0:0:1)\}. \]
By composing $\PP^2 \ttto S$ with $\phi$ we get a rational map
\[ \tau: \PP^2 \ttto \PP^1, \quad (x:y:z) \mapsto (x^4 : y^4 + xz^3) \]
that is also undefined only at $(0:0:1)$.
The members of the pencil of curves on $\PP^2$ induced by $\tau$ are up to isomorphism the fibres of $\phi: S \to \PP^1$.

Local computations show that the only singular point $P$ on the total space $S$ can be resolved by blowing up the surface eight times over $P$.
This produces a smooth projective surface $\wt S$ and fifteen smooth rational curves 
$E_1^{(1)}, E_2^{(1)}, \dots, E_1^{(7)}, E_2^{(7)}, E^{(8)} \su \wt S$ that are contracted to $P$ by the birational morphism $\wt S \to S$.
The resulting fibration
\[ f: \wt S \tto S \overset\phi\tto \PP^1 \]
is proper and flat, and its fibres over the points $(1:c)$ coincide with those of $\phi$.
Its bad fibre, i.e., the fibre over the point $(0:1)$, is a linear combination of smooth rational curves
\begin{equation}\label{2022_01_26_14:35}
    \begin{aligned}
    f^*(0:1) &= 4E + 3 E_1^{(1)} + E_2^{(1)} + 6 E_1^{(2)} + 2 E_2^{(2)} + 9 E_1^{(3)} + 3 E_2^{(3)} + 12 E_1^{(4)} + 4 E_2^{(4)}  \\
    &\qquad + 11 E_1^{(5)} + 5 E_2^{(5)} + 10 E_1^{(6)} + 6 E_2^{(6)} + 9 E_1^{(7)} + 7 E_2^{(7)} + 8 E_8^{(8)}
    \end{aligned}
\end{equation}
that intersect transversely according to the Coxeter-Dynkin diagram in Figure~\ref{2020_11_25_23:58}.
In this diagram the vertex $E$ denotes the strict transform of the curve $\phi^{-1}(0:1)$, 
and the dashed line means that the strict transform $H$ of the horizontal curve $(0:0:1) \times \PP^1 \subset S$ intersects the bad fibre $f^*(0:1)$
transversely at the component $E_2^{(1)}$ but does not actually belong to $f^*(0:1)$.
\begin{figure}[h]
\centering
\begin{tikzpicture}[line cap=round,line join=round,x=0.7cm,y=0.7cm]
\draw [line width=1.2pt] (-7.,0.)-- (7.,0.);
\draw [line width=1.2pt, dashed] (7.,0.)-- (8.,0.);
\draw [line width=1.2pt] (-4.,0.)-- (-4.,-1.);
\begin{scriptsize}
\draw [fill=black] (-7,0) circle (2pt);
\draw[color=black] (-7,0.5) node {$E^{(1)}_1$};    
\draw [fill=black] (-6,0) circle (2pt);
\draw[color=black] (-6,0.5) node {$E^{(2)}_1$};
\draw [fill=black] (-5,0) circle (2pt);
\draw[color=black] (-5,0.5) node {$E^{(3)}_1$};
\draw [fill=black] (-4,0) circle (2pt);
\draw[color=black] (-4,0.5) node {$E^{(4)}_1$};
\draw [fill=black] (-4.,-1.) circle (2pt);
\draw[color=black] (-4.,-1.5) node {$E$};
\draw [fill=black] (-3,0) circle (2pt);
\draw[color=black] (-3,0.5) node {$E^{(5)}_1$};
\draw [fill=black] (-2.,0.) circle (2pt);
\draw[color=black] (-2.,.5) node {$E^{(6)}_1$};
\draw [fill=black] (-1.,0.) circle (2pt);
\draw[color=black] (-1.,.5) node {$E^{(7)}_1$};
\draw [fill=black] (0.,0.) circle (2pt);
\draw[color=black] (0.,.5) node {$E^{(8)}$};
\draw[color=black] (0.,-1.) ;
\draw [fill=black] (1.,0.) circle (2pt);
\draw[color=black] (1.,0.5) node {$E^{(7)}_2$};
\draw [fill=black] (2.,0.) circle (2pt);
\draw[color=black] (2.,.5) node {$E^{(6)}_2$};
\draw [fill=black] (3.,0.) circle (2pt);
\draw[color=black] (3.,0.5) node {$E^{(5)}_2$};
\draw [fill=black] (4.,0.) circle (2pt);
\draw[color=black] (4,.5) node {$E^{(4)}_2$};
\draw[color=black] (4.,-1.) ;
\draw [fill=black] (5.,0.) circle (2pt);
\draw[color=black] (5.,0.5) node {$E^{(3)}_2$};
\draw [fill=black] (6,0) circle (2pt);
\draw[color=black] (6,0.5) node {$E^{(2)}_2$};
\draw [fill=black] (7,0) circle (2pt);
\draw[color=black] (7,0.5) node {$E^{(1)}_2$};
\draw [fill=black] (8,0) circle (2pt);
\draw[color=black] (8,0.5) node {$H$};
\end{scriptsize}
\end{tikzpicture}
\caption{Dual diagram of the bad fibre $f^*(0:1)$}\label{2020_11_25_23:58}
\end{figure}

\noindent Because a fibre meets its components with intersection number zero, 
we can determine from \eqref{2022_01_26_14:35} the self-intersection number of each component of $f^*(0:1)$, namely
\[ E\cdot E = -3, \qquad E^{(i)}_j \cdot E^{(i)}_j = -2 \:\text{ for each $i,j$}, \qquad E^{(8)}\cdot E^{(8)} = -2. \]
In particular, the point $P\in S$ is a Du Val singularity of type $A_{15}$.
As the bad fibre $f^*(0:1)$ 
contains no curves of self-intersection $-1$, that is, as the fibres of $f$ do not contain curves that are contractible according to Castelnuovo's contractibility criterion, we deduce that the smooth projective surface $\wt S$ is relatively minimal over $\PP^1$.
Consequently, the fibration $\wt S \to \PP^1$ is a relatively minimal regular model, and hence by a theorem of Lichtenbaum--Shafarevich \cite[Chapter~9, Theorem~3.21 and Corollary~3.24]{Liu02} the (unique) minimal regular model of the function field $F|K=k(S)|k(\PP^1)$.

However, the smooth surface $\wt S$ is not relatively minimal over $\Spec(k)$, because the birational morphism $\wt S \to S \to \PP^2$ is not an isomorphism.
More precisely, the previous paragraphs show that $\wt S \to \PP^2$ restricts to an isomorphism
\[ \wt S \setminus (E_1^{(1)} \cup E_1^{(2)} \cup \dots \cup E_2^{(1)} \cup H) \overset\sim\tto \PP^2 \setminus \{(0:0:1)\}, \]
and that its fibre over the point $(0:0:1)$ comprises the sixteen smooth rational curves $E_1^{(1)}$, $E_1^{(2)}$, $\dots$, $E_2^{(1)}$, $H$.
In light of \cite[Theorem~4.10, p.\,265]{Shaf13a}, this implies that
one of these curves must be contractible, i.e., 
the horizontal curve $H \su \wt S$ is contractible, i.e., $H\cdot H = -1$.
Therefore, if we blow down successively the curves $H$, $E_2^{(1)}$, $E_2^{(2)}$, $\dots$, $E_1^{(1)}$ then we obtain a surface that is isomorphic to the projective plane.

We summarize the above discussion in the following theorem.

\begin{thm}\label{2022_02_11_12:20}        
    The fibration $f:\wt S \to \PP^1$ is the minimal 
    regular
    model of the fibration $\phi: S\to \PP^1$.
    Its fibres over the points $(1:c)$ coincide with the corresponding fibres of $\phi$, while its fibre over the point $(0:1)$ is a linear combination of smooth rational curves as in \eqref{2022_01_26_14:35}, which intersect transversely according to the Coxeter-Dynkin diagram in Figure~\ref{2020_11_25_23:58}.
    
    The strict transform $H \subset \wt S$ of the curve $(0:0:1)\times \PP^1\su S$ is a horizontal smooth rational curve of self-intersection $-1$.
    If we blow down successively the curves $H$, $E_2^{(1)}$, $E_2^{(2)}$, $\dots$, $E_1^{(2)}$, $E_1^{(1)}$, then we obtain a smooth surface isomorphic to the projective plane.
\end{thm}

As the function field $F|K=k(S)|k(\PP^1)$ is an inseparable extension of degree 2 of the quasi-elliptic Frobenius pullback $F_1|K$, it is clear that the fibration $f:\wt S \to \PP^1$ is a purely inseparable covering of degree 2 of a quasi-elliptic fibration.
Let us describe this covering.

We consider the fibration by cuspidal cubics given by the projection morphism
\[ \phi':S'\tto \PP^1, \]
where 
\[ S' = V(T_0( U V^2 + W^3 ) + T_1 U^3) \su \PP^2 \times \PP^1 \]
is a surface with a unique singularity at $P':=((0:1:0),(0:1))$.
The fibre over the point $(0:1)$ is a triple line, and so it is the bad fibre of the fibration.
As the projection morphism $S'\to \PP^2$ is birational, the surface $S'$ is rational.
The minimal regular model 
\[ f':\wt S' \tto S' \overset{\phi'}\tto \PP^1 \]
is constructed by resolving the singular point $P'$ on $S'$. 
The quasi-elliptic surface $\wt S'$ has an arrangement of smooth rational curves that intersect according to Figure~\ref{2023_05_18_11:00},
where $A$ and $H'$ denote the strict transforms of the bad fibre $\phi'^{-1}(0:1) \su S'$ and the horizontal curve $(0:1:0)\times \PP^1 \su S'$ respectively.

\begin{figure}[h]
\centering
\begin{tikzpicture}[line cap=round,line join=round,x=0.7cm,y=0.7cm]
\draw [line width=1.2pt] (-7.,0.)-- (0.,0.);
\draw [line width=1.2pt] (-8.,0.)-- (-7.,0.);
\draw [line width=1.2pt] (-6.,0.)-- (-6.,-1.);
\begin{scriptsize}
\draw [fill=black] (-8,0) circle (2pt);
\draw[color=black] (-8,0.5) node {$A_1^{(1)}$}; 
\draw [fill=black] (-7,0) circle (2pt);
\draw[color=black] (-7,0.5) node {$A_1^{(2)}$}; 
\draw [fill=black] (-6,0) circle (2pt);
\draw[color=black] (-6,0.5) node {$A_1^{(3)}$}; 
\draw [fill=black] (-6.,-1.) circle (2pt);
\draw[color=black] (-6.,-1.5) node {$A$};
\draw [fill=black] (-5,0) circle (2pt);
\draw[color=black] (-5,0.5) node {$A_1^{(4)}$}; 
\draw [fill=black] (-4,0) circle (2pt);
\draw[color=black] (-4,0.5) node {$A_2^{(4)}$}; 
\draw [fill=black] (-3,0) circle (2pt);
\draw[color=black] (-3,0.5) node {$A_2^{(3)}$}; 
\draw [fill=black] (-2.,0.) circle (2pt);
\draw[color=black] (-2.,.5) node {$A_2^{(2)}$}; 
\draw [fill=black] (-1.,0.) circle (2pt);
\draw[color=black] (-1.,.5) node {$A_2^{(1)}$}; 
\draw [fill=black] (0.,0.) circle (2pt);
\draw[color=black] (0.,.5) node {$H'$}; 
\draw[color=black] (0.,-1.) ;
\end{scriptsize}
\end{tikzpicture}
\caption{Dual diagram of curves in $\wt S'$}\label{2023_05_18_11:00}
\end{figure}
\noindent 
These curves have self-intersection numbers
$ A\cdot A=A_i^{(j)} \cdot A_i^{(j)} = -2$, $H' \cdot H' = -1$.
The only bad fibre 
$f'^*(0:1)$ of the quasi-elliptic fibration $f':\wt S \to \PP^1$
is of type $\tilde E_8$ (see \cite[Theorem~4.1.4 and Corollary~4.3.22]{CDL23}), 
or more precisely
\begin{equation*}
    \begin{aligned}
    f'^*(0:1) &= 2A_1^{(1)} + 4 A_1^{(2)} + 6 A_1^{(3)} + 5 A_1^{(4)} + 4 A_2^{(4)} + 3 A_2^{(3)} + 2 A_2^{(2)} + A_2^{(1)} + 3A.
    \end{aligned}
\end{equation*}
The exceptional fibre of the blowup morphism $\wt S'\to S' \to \PP^2$, i.e., the fibre over the point $(0:1:0)$, comprises the curves $H'$, $A_2^{(1)}$, $\dots$, $A_1^{(1)}$.

To describe the covering of the quasi-elliptic fibration $\wt S' \to \PP^1$ by our fibration $\wt S \to \PP^1$ we consider the rational map
\[ S \ttto S', \quad ((x:y:z),(t_0:t_1)) \mapsto ((x^2:y^2:xz),(t_0 : t_1)), \]
which is undefined only at the singular point \mbox{$P=((0:0:1),(0:1))$}.
The induced rational map
\[ \wt S \ttto \wt S' \]
is defined everywhere except at the point $Q\in E$ that lies above the point $((0:1:0),(0:1)) \in S$.
Over each point $(1:c)$ of the base $\PP^1$ the fibre in $\wt S$ covers the fibre in $\wt S'$ according to the rule
$ (x:y:z) \mapsto (x^2 : y^2 : xz) $.
The curves in Figure~\ref{2020_11_25_23:58} are taken into the curves in Figure~\ref{2023_05_18_11:00} as follows: the curves 
$E_1^{(2)}$, $E_1^{(4)}$, $E_1^{(6)}$, $E^{(8)}$, $E_2^{(6)}$, $E_2^{(4)}$, $E_2^{(2)}$, $H$ 
are mapped isomorphically to the curves 
$A$, $A_1^{(3)}$, $A_1^{(4)}$, $A_2^{(4)}$, $A_2^{(3)}$, $A_2^{(2)}$, $A_2^{(1)}$, $H'$ 
respectively, and
the curves 
$E_i^{(j)}$ ($j$ odd)
are contracted to points.
If we resolve the indeterminacy locus of $\wt S \ttto \wt S'$ by blowing up the point $Q$, an exceptional curve $L$ appears, which covers $A_1^{(1)}$ isomorphically,
and in addition the curve $E$ becomes a purely inseparable cover of degree $2$ of the curve $A_1^{(2)}$.
All this is illustrated in the diagram below.

\begin{figure}[h]

\centering
\begin{tikzpicture}[line cap=round,line join=round,x=0.8cm,y=0.8cm]



\begin{scriptsize}

\draw (-8,0 - .6) -- (-8,2 + .6) ;  \draw[color=black] (-8,0 - .6) node[below] {$E_1^{(1)}$}; 
\draw (-6,0 - .6) -- (-6,2 + .6) ;  \draw[color=black] (-6,0 - .6) node[below] {$E_1^{(3)}$}; 
\draw (-4,0 - .6) -- (-4,2 + .6) ;  \draw[color=black] (-4,0 - .6) node[below] {$E_1^{(5)}$}; 
\draw (-2,0 - .6) -- (-2,2 + .6) ;  \draw[color=black] (-2,0 - .6) node[below] {$E_1^{(7)}$}; 
\draw (0,0 - .6) -- (0,2 + .6) ;  \draw[color=black] (0,0 - .6) node[below] {$E_2^{(7)}$}; 
\draw (2,0 - .6) -- (2,2 + .6) ;  \draw[color=black] (2,0 - .6) node[below] {$E_2^{(5)}$}; 
\draw (4,0 - .6) -- (4,2 + .6) ;  \draw[color=black] (4,0 - .6) node[below] {$E_2^{(3)}$}; 
\draw (6,0 - .6) -- (6,2 + .6) ;  \draw[color=black] (6,0 - .6) node[below] {$E_2^{(1)}$}; 

\draw (-8- .6,2) -- (-6 + .6,2) ;  \draw[color=black] (-7,2 + .0) node[above] {$E_1^{(2)}$}; 
\draw (-4- .6,2) -- (-2 + .6,2) ;  \draw[color=black] (-3,2 + .0) node[above] {$E_1^{(6)}$}; 
\draw (0- .6,2) -- (2 + .6,2) ;  \draw[color=black] (1,2 + .0) node[above] {$E_2^{(6)}$}; 
\draw (4- .6,2) -- (6 + .6,2) ;  \draw[color=black] (5,2 + .0) node[above] {$E_2^{(2)}$}; 
\draw (-6- .6,0) -- (-4 + .6,0) ;  \draw[color=black] (-6- .6,0) node[left] {$E_1^{(4)}$}; 
\draw (-2- .6,0) -- (0 + .6,0) ;  \draw[color=black] (-1,0 + .0) node[above] {$E^{(8)}$}; 
\draw (2- .6,0) -- (4 + .6,0) ;  \draw[color=black] (3,0 + .0) node[above] {$E_2^{(4)}$}; 
\draw (6- .6,0) -- (8 + .6,0) ;  \draw[color=black] (7 + .4,0 + .0) node[above] {$H$}; 

\draw (-8 -.6,-4 + .6/2) -- (-6 + .6,-5 - .6/2) ;  \draw[color=black] (-7 -.5,-4.5 +.5/2 + .0) node[below] {$A$};
\draw (-4 -.6,-4 + .6/2) -- (-2 + .6,-5 - .6/2) ;  \draw[color=black] (-3 +.3,-4.5 -.3/2 + .1) node[above] {$A_1^{(4)}$}; 
\draw (0 -.6,-4 + .6/2) -- (2 + .6,-5 - .6/2) ;  \draw[color=black] (1 +.3,-4.5 -.3/2 + .1) node[above] {$A_2^{(3)}$}; 
\draw (4 -.6,-4 + .6/2) -- (6 + .6,-5 - .6/2) ;  \draw[color=black] (5 +.3,-4.5 -.3/2 + .1) node[above] {$A_2^{(1)}$}; 
\draw (-4 + .6,-4 + .6/2) -- (-6 - .6,-5 - .6/2) ;  \draw[color=black] (-4 + .6,-4 + .6/2) node[above] {$A_1^{(3)}$}; 
\draw (0 + .6,-4 + .6/2) -- (-2 - .6,-5 - .6/2) ;  \draw[color=black] (-1 +.3,-4.5 +.3/2 - .0) node[below] {$A_2^{(4)}$}; 
\draw (4 + .6,-4 + .6/2) -- (2 - .6,-5 - .6/2) ;  \draw[color=black] (3 +.3,-4.5 +.3/2 - .0) node[below] {$A_2^{(2)}$}; 
\draw (8 + .6,-4 + .6/2) -- (6 - .6,-5 - .6/2) ;  \draw[color=black] (7 +.5,-4.5 +.5/2 - .0) node[below] {$H'$}; 

\draw (-5 - 1.3,-4.5 + 1.3/2) -- (-5 + 1.3,-4.5 - 1.3/2) ;  \draw[color=black] (-5 -.8,-4.5 +.8/2 +.1 ) node[above] {$A_1^{(2)}$}; 
\draw (-5 + 1.3,-5.5 + 1.3/2) -- (-5 - 1.3,-5.5  -1.3/2) ;  \draw[color=black] (-5 - .6,-5.5 - .6/2) node[below] {$A_1^{(1)}$}; 

\draw (-5 -.5,0 + 1.4 - .8) -- (-5 +.5,0 -1.4 - .8) ;  \draw[color=black] (-5 -.5,0 + 1.4 - .8) node[above] {$E$};

\draw (-6- .6,0 - 1.6) -- (-4 + .6,0 - 1.6) ;  \draw[color=black] (-6- .6,0 - 1.6) node[left] {$L$};

\draw [fill=black] (-8,-4) circle (1.6pt);
\draw [fill=black] (-6,-5) circle (1.6pt);
\draw [fill=black] (-4,-4) circle (1.6pt);
\draw [fill=black] (-2,-5) circle (1.6pt);
\draw [fill=black] (0.,-4) circle (1.6pt);
\draw [fill=black] (2,-5) circle (1.6pt);
\draw [fill=black] (4,-4) circle (1.6pt);
\draw [fill=black] (6,-5) circle (1.6pt);

\end{scriptsize}
\end{tikzpicture}
\end{figure}




\end{document}